\documentclass[11pt]{article}
\usepackage[utf8]{inputenc}
\usepackage{mystyle}

\begin{document}

\title{\huge Robust High-dimensional Tuning Free Multiple Testing\thanks{Supported by NSF grants DMS-2210833, DMS-2053832, DMS-2052926 and ONR grant N00014-22-1-2340}}

\author{Jianqing Fan\qquad Zhipeng Lou\qquad Mengxin Yu}


\maketitle

\begin{abstract}
A stylized feature of high-dimensional data is that many variables have heavy tails, and robust statistical inference is critical for valid large-scale statistical inference. Yet, the existing developments such as Winsorization, Huberization and median of means require the bounded second moments and involve variable-dependent tuning parameters, which hamper their fidelity in applications to large-scale problems. To liberate these constraints, this paper revisits the celebrated Hodges-Lehmann (HL) estimator for estimating location parameters in both the one- and two-sample problems, from a non-asymptotic perspective. Our study develops Berry-Esseen inequality and Cram\'{e}r type moderate deviation for the HL estimator based on newly developed non-asymptotic Bahadur representation, and builds data-driven confidence intervals via a weighted bootstrap approach. These results allow us to extend the HL estimator to large-scale studies and propose \emph{tuning-free} and \emph{moment-free} high-dimensional inference procedures for testing global null and for large-scale multiple testing with false discovery proportion control. It is convincingly shown that the resulting tuning-free and moment-free methods control false discovery proportion at a prescribed level.  The simulation studies lend further support to our developed theory.
\end{abstract}


\section{Introduction}
Large-scale, high-dimensional data with rich structures have  been widely collected in almost all scientific disciplines and humanities, thanks to the advancements of modern technologies. Massive developments have been made in statistics over the past two decades on extracting valuable information from these high dimensional data; see~\cite{buHLmann2011statistics, hastie2009elements, hastie2015statistical, wainwright2019high, fan2020statistical, chen2021spectral} for a detailed account and references therein.

Despite convenience for theoretical analysis, the sub-Gaussian tails condition is  not realistic in many applications that involve high-dimensional variables. For instance, it is well known that heavy-tailed distributions is a stylized
feature for financial returns and macroeconomic variables~\citep{cont2001empirical,stock2002macroeconomic,fan2017elements,fan2021shrinkage}. Therefore, tools designed for sub-Gaussian data can lead to erroneous scientific conclusions. Asking thousands of gene expressions to have all sub-Gaussian tails is a mathematical dream, not a reality that data scientists face. For example, comparing gene expression profiles between various cell sub-populations, especially after treatments and therapies, is an essential statistical task~\citep{nagalakshmi2008transcriptional, shendure2008next, wang2009rna, li2012normalization, li2013finding, gupta2014transcriptome, finotello2015measuring}. However, it is unrealistic to hope all thousands of gene expressions have sub-Gaussian distributions: outliers in non-sub-Gaussian distributions can have a significant impact on nonrobust procedures and lead to many false positives and negatives~\citep{gupta2014transcriptome, wang2015high}. The situation also arises for inferences using functional magnetic resonance imaging (fMRI) data since the data do not conform to the assumed Gaussian distribution~\citep{eklund2016cluster}. These practical challenges demand for developing efficient and reliable robust inference methods.

Recently, robust statistical methods have gained popularity as a mean of resolving outliers and heavy-tailed noises.  Many preceding arts have taken a significant stride toward effective statistical estimation under heavy-tailed distributions.  For instance, aiming at dealing with the heavy-tailed noise contamination, the Huber regression is proposed~\citep{huber1973robust}, and subsequent publications along these lines include~\citet{yohai1979asymptotic, mammen1989asymptotics, he1996general, he2000parameters}, where the asymptotic properties of the Huber estimator have been thoroughly investigated. From a non-asymptotic perspective, the Huber-type estimator was recently revisited by~\citet{sun2020adaptive}, in which the authors propose an adaptive Huber regression method and establish its non-asymptotic deviation bounds by only requiring finite $(1 + \delta)$-th moment of the noise with any $\delta > 0$.  Moreover, using a similar idea for making the correspondent $M$-estimator insensitive to extreme values, \citet{catoni2012challenging} developed a novel approach through minimizing a robust empirical loss. It is demonstrated that the estimator has exponential concentration around the true mean and enjoys the same statistical rate as the sample average for sub-Gaussian distributions when the population only has a bounded second moment. \citet{brownlees2015empirical} further investigates empirical risk minimization based on the robust estimator proposed in~\citet{catoni2012challenging}. Additionally, the so-called~\emph{median of means} strategy, which can be traced back to~\citet{nemirovskij1983problem}, is another successful method for handling heavy-tailed distributions. By only requiring bounded second moment, it achieves the sub-Gaussian type of concentration around the population mean parameter. \citet{minsker2015geometric} and~\citet{hsu2016loss} further generalize this idea to multivariate cases. Moreover, there also exists a series of works that focus on solving the issue caused by heavy-tailed noises using quantile-based robust estimation; see~\citet{Arcones1995, koenker2001quantile, Belloni2011, fan2014adaptive, zheng2015globally} for more details. Furthermore, recently, in~\citet{fan2021shrinkage, yang2017high, fan2022understanding}, under heavy-tailed contamination, they proposed a novel principle by simply truncating or shrinking the response variables appropriately to achieve sub-Gaussian rates, and they only require bounded second moment of the measurements. Additionally, the aforementioned methodologies can also be applied to a wide range of problems, such as matrix sensing, matrix completion, robust PCA, factor analysis, and neural networks. For interested readers, we refer to~\citet{minsker2015geometric, hsu2016loss, fan2017estimation, loh2017statistical, minsker2018sub, goldstein2018structured,wang2020tuning, fan2022latent, wang2022robust, fan2022noise} for more details.

While many effective solutions have been developed to address the problem of heavy-tailedness, these solutions still have some potential shortcomings. In specific, the developments call for the second moments to be bounded and are primarily based on shrinkage data, Huber-type of loss, median of means~\citep{huber1973robust, nemirovskij1983problem, catoni2012challenging, fan2021shrinkage}. More critically, Huberization, Winsorization and sample splitting introduce additional tuning parameters and these tuning paprameters should be variable-dependent that makes large-scale applications difficult and damages the fidelity of empirical results . Although, the quantile estimators~\citep{koenker2001quantile} such as the median and the Hodge-Lehmann (HL) estimator~\citep{Hodges1963} can eliminate the restriction on moment conditions and tuning parameters selection, the empirical median is often less efficient and requires stronger distribution assumptions. In addition, the existing literature on the HL estimator focuses mainly on low-dimensional asymptotic analysis and can not be applied to large-scale inferences.


In this paper, we revisit the celebrated HL estimator~\citep{Hodges1963} and conduct non-asymptotic and large-scale theoretical studies for both one-sample and two-sample problems. For one-sample location estimation in the univariate case, we let $X_{1}, \ldots, X_{n} \in \RR$ be independent and identically distributed (i.i.d.)~random variables with
\begin{align}
\label{x_univ}
    X_{i} = \theta + \xi_{i}, \enspace i = 1, \ldots, n,
\end{align}
where $\theta$ represents the location parameter of interest and $\xi_{1}, \ldots, \xi_{n}$ are i.i.d.~random variables drawn from some unknown distribution. In this scenario, the HL estimator of $\theta$ is defined by
\begin{align}
\label{eq_def_median_sum_one_sample}
    \hat{\theta} = \mathrm{median}\left\{\frac{X_{i} + X_{j}}{2} : 1\leq i < j\leq n\right\}. 
\end{align}
By assuming the pseudomedian of $\xi_{1}$ to be zero, we derive~\emph{non-asymptotic} Bahadur representation of $\hat\theta.$ To the best of our knowledge, this is the first study of its kind on the non-asymptotic expansion of the HL estimator. From there, we also establish the Berry-Essen bound and moderate deviation for $\hat\theta$ in the widest range. Furthermore, as there are multiple unknowns in the asymptotic distribution of $\hat\theta$, including the density function of $\xi_{1}$ and the unknown location parameter $\theta$, we then propose a weighted bootstrap approach to construct confidence intervals for $\theta$ based on data.  These results and methods are essential for the large-scale inference.

In addition to the study of one-sample problem, two-sample location shift problems arise frequently in many scientific studies, including choosing genes that are expressed differently in normal and injured spinal cord, determining the effects of treatment between treated and control groups, finding change points, etc.  
 To this end, we  let $Y_{1}, \ldots, Y_{m} \in \RR$ be another independent sample of i.i.d.~random variables satisfying 
\begin{align}\label{y_univ}
    Y_{j} = \theta^{\circ} + \varepsilon_{j}, \enspace j = 1, \ldots, m. 
\end{align}
The primary goal is to conduct statistical inference for $\Theta = \theta - \theta^{\circ}$.
 In the sequel, following~\citet{Hodges1963}, the two-sample HL  estimator for $\Theta$ is given by 
\begin{align}
\label{eq_HL _estimator_two}
    \hat{\Theta} = \mathrm{median}\{X_{i} - Y_{j} : i = 1, \ldots, n;\, j = 1, \ldots, m\}.
\end{align}
Instead of assuming the noises are generated from the same distribution \citep{Hodges1963}, we only require median$(\xi_{1} - \varepsilon_{1}) = 0$, which is more general and allows random noises to have different distributions. In a similar vein, we establish the non-asymptotic expansion of $\hat\Theta$, investigate its asymptotic distributions, and calculate the confidence interval via bootstrap techniques.  Again, the techniques and results developed can be applied to large-scale multiple testing problems.

There is a rich literature on large-scale multiple testing problems for location parameters~\citep{Benjamini1995, storey2002direct, storey2003positive, genovese2004stochastic, ferreira2006benjamini, chi2007performance, blanchard2009adaptive}. However, most of these works assume the noise distributions sub-Gaussian. Moving away for sub-Gaussian assumptions, \cite{Fan2019} propose estimating mean vector via minimizing the Huber type loss and perform the false discovery proportion (FDP) control. However, leveraging Huber type estimators necessitates moment limits and introduces tuning parameters, making it hard to be applied in large-scale inference, as the tuning parameters should ideally be variable-dependent. Additionally, while the HL estimator enjoys tuning-free and moment-free qualities in the univariate setting, its behavior in high dimensions is largely unknown. To this end, in this research, we further expand the HL estimator to high-dimensional regimes prompted by the lack of tuning free large-scale multiple testing problems for heavy-tailed distributions.


In specific, we assume $\bX_{i} = \btheta + \bxi_{i}$, $i \in [n]$ and $\bY_{j} = \btheta^\circ + \bvarepsilon_{j},j\in[m]$, where $\btheta = (\theta_{1}, \ldots, \theta_{p})^{\top}$ and $\btheta^\circ = (\theta_{1}^\circ, \ldots, \theta_{p}^\circ)^{\top}$ are~\emph{p}-dimensional vectors of unknown parameters and random noises $\bxi_{i}, i\in[n], \bvarepsilon_j, j\in[m]$. Let $\bTheta=\btheta-\btheta^\circ$. For both one- and two-sample problems, we propose a carefully constructed Gaussian multiplier bootstrap to test global null hypotheses
\begin{align}
\label{eq_high_dimensional_one_sample_testing}
    H_{0} : \theta_{\ell} \,\,\textrm{or}\,\,\Theta_\ell= 0 \mbox{ for all } \ell \in [p] \enspace \mathrm{versus} \enspace H_{1} : \theta_{\ell} \,\,\textrm{or}\,\,\Theta_\ell\neq 0\mbox{ for some } \ell \in [p],
\end{align}
by extending the HL estimator to high-dimensional regimes. When the null hypothesis above is rejected, we then perform multiple testing, allowing weakly dependent measurements, and efficiently control the FDP. Compared with existing literature~\citep{Liu2014, Fan2019}, our procedures do not involve any  tuning parameters and moment conditions for testing global null and large-scale multiple testing. These theoretical finds are further supported by  exhaustive numerical studies.

The main contributions of the paper can be summarized as follows:
\begin{itemize}
    \item The existing studies on the HL estimator mainly focus on its asymptotic behavior, which is too weak for high-dimensional applications. In practice, however, it is crucial to understand the HL estimator's performance under finite sample, especially in high-dimensional and large-scale experiments. For this purpose, we first derive the non-asymptotic expansions of the HL  estimators for both one-sample and two-sample problems. 

    \item With the non-asymptotic expansions of the HL estimators, for both one- and two-sample problems, we derive its Berry-Essen type bounds and Cram\'{e}r type moderate deviations, with the widest range. To deal with unknown components in the distribution, we further develop the weighted bootstrap to build data-driven confidence intervals. In addition, we also furnish the non-asymptotic analysis of the bootstrap estimator.
    
    \item Existing work on large-scale testing with heavy-tail errors typically involves additional tuning parameters and the moment conditions. In order to address these issues, we generalize the HL estimator to large-scale studies and propose tuning-free and moment-free high-dimensional testing procedures. Additionally, we develop bootstrap methods for calculating critical values for large-scale applications.  We show that the resulting false discovery proportion is well controlled.
\end{itemize}

\subsection{Roadmap}
In~\S\ref{sec:nonasym}, we first set up the model and introduce basic settings. We then derive both one-sample and two-sample non-asymptotic expansions of the HL estimator, its Berry-Esseen bound and moderate deviations. In addition, as the asymptotic distribution of the estimator involves unknown quantities, in~\S\ref{bootstrap}, we conduct multiplier bootstrap to construct valid data-driven confidence intervals. Moreover, \S\ref{large-scale-one} is devoted to extending the HL estimator to large-scale multiple testing problems. \S\ref{sec:num} contains comprehensive numerical studies to verify theoretical
results. Finally we conclude the paper with some discussions in~\S\ref{sec:conc}. All the proofs are deferred to the appendix.

\subsection{Notation}
For any integer $m$, we use $[m]$ to denote the set $[m] = \{1, 2, \ldots, m\}$. For any function $h : \mathbb{R} \to \mathbb{R}$, we denote $\|h\|_{\infty} = \sup_{z \in \RR}|h(z)|$. Throughout this paper, we use $C, C_{1}, C_{2}, \ldots$ to denote universal positive constants whose values may vary at different places. We use $\mathbb{I}\{\cdot\}$ to denote the indicator function. For any set $A$, we use $|A|$ to denote its cardinality. For two positive sequences $\{a_n\}_{n\ge 1}$ and $\{b_n\}_{n\ge 1}$, we write $a_n=\cO(b_n)$ or $a_n\lesssim b_n$ if there exists a positive constant $C$ such that $a_n\le C\cdot b_n$ and we write $a_n=o(b_n)$ if $a_n/b_n\rightarrow 0$. In addition, we define the pseudomedian of a distribution $F$ to be the median of the distribution of $(Z_1+Z_2)/2,$ where $Z_1$ and  $Z_2$ are independent, each with the same distribution $F$. Moreover, for any distribution $F$ and constant $c$, we let $c\cdot F$ represent the distribution of the random variable $c\cdot X$, where $X$ is the random variable drawn from $F.$



\section{Estimation and Inference}\label{sec:nonasym}
This section is devoted to studying the non-asymptotic expansions of the Hodges-Lehmann estimator and conducting statistical estimation and inference for population location shift parameters. For both one-sample and two-sample problems, the theoretical properties, which are needed for large-scale inferences, are presented in the following sections.

\subsection{One-sample Problem}
\label{framework}
Let $X_{i} = \theta + \xi_{i}$, $i \in [n]$, be i.i.d.~real-valued random variables, where $\theta \in \RR$ is the unknown location parameter of interest and $\xi_{1}, \ldots, \xi_{n}$ are i.i.d.~random variables drawn from some unknown distribution. It is assumed that $\xi_{1}$ has a pseudomedian~\citep{Hoyland1965} of zero, throughout this section. As a consequence, letting $U(t) = \mathbb{P}\{(X_{1} + X_{2})/2 \leq t\}$, it holds that $\theta = \inf\{t \in \mathbb{R} : U(t) \geq 1/2\}$. The HL estimator~\citep{Hodges1963} of $\theta$ is given by the median of all~\emph{Walsh averages} of the observations $X_{1}, \ldots, X_{n}$, namely, 
\begin{align}
\label{eq_def_median_sum_one_sample}
    \hat{\theta} = \mathrm{median}\{(X_{i} + X_{j})/2 : i \neq j \in [n]\}. 
\end{align}
Equivalently, if we define the~\emph{U}-process $U_{n}(t) = \{n(n - 1)\}^{-1} \sum_{i \neq j \in [n]} \mathbb{I}\{(X_{i} + X_{j})/2 \leq t\}$, the HL estimator $\hat{\theta}$ in~\eqref{eq_def_median_sum_one_sample} can also be expressed as the sample median of the process $U_{n}(t)$, namely,
\begin{align}
\label{eq_def_mu_hat_quantile}
    \hat{\theta} = \inf\{t \in \RR : U_{n}(t) \geq 1/2\}. 
\end{align}
Let $F(t) = \mathbb{P}(\xi_{1} \leq t)$ denote the cumulative distribution function of $\xi_{1}$ and $f(t) = F'(t)$ be its density function. We then present the non-asymptotic Bahadur representation of $\hat\theta$ in the following theorem.

\begin{theorem}
\label{Theorem_quantile_consistency_one}
Assume that there exist positive constants $c_{0}$ and $\kappa_{0}$ such that $\inf_{|\delta| \leq c_{0}} U'(\theta + \delta) \geq \kappa_{0}$. Then for any $z > 0$, we have 
\begin{align}
\label{eq_consistency_mu}
    \PP(|\hat{\theta} - \theta| > z) \leq 2 \exp\{-n\kappa_{0}^{2}(z\wedge c_{0})^{2}\}.
\end{align}
Furthermore, assume that $\sup_{z \in \RR} |f(z)| < \infty$ and there exist positive constants $c_{1}$ and $\kappa_{1}$ such that $\sup_{|\delta| \leq c_{1}} |U''(\theta + \delta)| \leq \kappa_{1}$. Then for any $z > 0$ such that $z = o(n)$, we have 
\begin{align}
\label{eq_Hoeffding_decomposition_Bahadur_representation}
    \mathbb{P}\bigg\{\bigg|\hat{\theta} - \theta - \frac{2}{n U'(\theta)}\sum_{i = 1}^{n}\bigg\{\frac{1}{2} - F(-\xi_{i})\bigg\}\bigg| > \frac{C_{1}(z \vee 1)}{n}\bigg\} \leq C_{2} \exp(-z), 
\end{align}
where $C_{1}, C_{2}$ are positive constants depending only on $c_{0}, \kappa_{0}, c_{1}, \kappa_{1}$ and $\|f\|_{\infty}$. 
\end{theorem}


We note that existing works mainly study the asymptotic distribution of quantiles of~\emph{U}-statistics instead of non-asymptotic ones~\citep{Arcones1996}. Asymptotic theory, however, is frequently less effective for theoretical studies in  high-dimensional statistics~\citep{wainwright2019high}. To fill in the blank, in Theorem~\ref{Theorem_quantile_consistency_one}, we present both the non-asymptotic deviation bound and linear approximation of the HL estimator $\hat\theta$. It is worth mentioning that the HL estimator $\hat{\theta}$ has sub-Gaussian tails without any moment constraints imposed on the noise $\xi_{1}$, whereas the Huber-type or winsorized estimator requires the existence of the second moment. Moreover, in contrast to Huber regression~\citep{Zhou2018, sun2020adaptive} or truncation~\citep{fan2021shrinkage}, which both require additional tuning parameters, HL-type estimation is tuning-free and thus more scalable.


Moreover, when the distribution of $\xi_{1}$ is symmetric around zero, $\theta$ reduces to the median of the distribution of $X$. In this scenario, the sample median $\hat{\theta}_{\mathrm{med}} = \mathrm{median}\{X_{1}, \ldots, X_{n}\}$ serves as a plausible alternative robust estimator for $\theta$. Under similar regularity conditions on the density function $f(t)$, the classical Bahadur representation for $\hat{\theta}_{\mathrm{med}}$ reveals that
\begin{align}
\label{eq_Bahadur_representation_one_sample_median}
    \PP\bigg\{\bigg|\hat{\theta}_{\mathrm{med}} - \theta - \frac{1}{n f(0)} \sum_{i = 1}^{n}\bigg(\frac{1}{2} - \mathbb{I}\{\xi_{i} \leq 0\}\bigg)\bigg| > \frac{C\log n}{n^{3/4}}\bigg\} \leq C n^{-c}
\end{align}
for any constant $c > 0$.
Compared with~\eqref{eq_Hoeffding_decomposition_Bahadur_representation}~(by taking $z = \cO(\log n)$), the linear approximation of HL estimator is much more accurate than that of the quantile estimator~\citep{Arcones1996}.

\subsubsection{Asymptotic Distribution}
In addition to estimation, statistical inference is also essential in real-world applications. To this end, with the developed non-asymptotic expansion at hand, we next present the asymptotic distribution of the HL estimator $\hat\theta$ in this section.

Let $\Phi(\cdot)$ denote the cumulative distribution function of standard normal random variable. The following theorem establishes a Berry-Esseen theorem for $\hat{\theta}$.

\begin{theorem}
\label{Theorem_Berry-Esseen_one}
Under the conditions of Theorem~\ref{Theorem_quantile_consistency_one}, we have 
\begin{align*}
    \sup_{z \in \RR}\bigg|\PP\bigg\{\frac{\sqrt{n}(\hat{\theta} - \theta)}{\sigma_{\theta}} \leq z\bigg\} - \Phi(z)\bigg| \leq \frac{C\log n}{\sqrt{n}},  
\end{align*}
where $C < \infty$ is a positive constant independent of $n$ and 
\begin{align}
\label{eq_sigma_variance_one}
    \sigma_{\theta}^{2} = \frac{4\Var\{F(-\xi_{1})\}}{\{U'(\theta)\}^{2}}.
\end{align} 
\end{theorem}

Theorem~\ref{Theorem_Berry-Esseen_one} establishes the asymptotic normality of $\hat\theta$. When the distribution of $\xi_{1}$ is symmetric around zero, the asymptotic variance above reduces to $\sigma_{\theta}^{2} = 1/[3\{U'(\theta)\}^{2}]$. Consequently, in view of~\eqref{eq_Bahadur_representation_one_sample_median} and ~\eqref{eq_sigma_variance_one}, the asymptotic relative efficiency (ARE) between the HL estimator $\hat{\theta}$ and the sample median $\hat{\theta}_{\mathrm{med}}$ is ${3\{U'(\theta)\}^{2}}/{4\{f(0)\}^{2}}$ \citep{Hodges1963}. A concrete example is given in Table~\ref{table_RE_one_sample}, where we summarize the ARE between $\hat{\theta}$ and $\hat{\theta}_{\mathrm{med}}$ for $\xi_{1} \sim t_{\nu}$.
\begin{table}[h]
    \centering
    \begin{tabular}{c|cccccc}\hline
    $\nu$  & $1$ & $2$ & $4$ & $8$ & $16$ & $\infty$ \\ \hline
    ARE     & $0.75$ & $1.04$ & $1.25$ & $1.37$ & $1.43$ & $1.50$ \\ \hline
    \end{tabular}
    \caption{Asymptotic relative efficiency between the HL estimator and the sample median}
    \label{table_RE_one_sample}
\end{table}
In particular, when $\nu\ge 2$, the HL estimator has a strictly smaller asymptotic variance than the sample median. The above example illustrates the effectiveness of the HL estimator over the quantile regression method.

Based on the non-asymptotic linear expansion in~\eqref{eq_Hoeffding_decomposition_Bahadur_representation}, we further derive the Cram\'{e}r-type moderate deviation to quantify the relative error of the normal approximation for $\hat\theta$ in the following theorem, which has important applications to large-scale inference~\citep{Fan2007, Liu2014, Xia2018, Zhou2018}.


\begin{theorem}
\label{Theorem_Cramer_one}
Let $\{\delta_{n}\}_{n \geq 1}$ be a sequence of positive numbers satisfying $\sqrt{n}\delta_{n} \to \infty$. Then, under the conditions of Theorem~\ref{Theorem_quantile_consistency_one}, we have
\begin{align}
\label{eq_moderate_deviation_mu_one_sample}
    \bigg|\frac{\PP\{\sqrt{n}(\hat{\theta} - \theta)/\sigma_{\theta} > z\}}{1 - \Phi(z)} - 1\bigg| \leq C\bigg\{\frac{1 + z^{3}}{\sqrt{n}} + (1 + z)\delta_{n} + 2(z\vee 1) \sqrt{2\pi} \exp\bigg(\frac{z^{2}}{2} - \sqrt{n}\delta_{n}\bigg)\bigg\},
\end{align}
uniformly for $0 < z \leq o(\delta_{n}^{-1} \wedge n^{1/4}\sqrt{\delta_{n}})$, where $C < \infty$ is a positive constant independent of $z$ and $n$. In particular, when $\delta_{n} \asymp n^{-1/6}$, we have 
\begin{align*}
    \bigg|\frac{\PP\{\sqrt{n}(\hat{\theta} - \theta)/\sigma_{\theta} > z\}}{1 - \Phi(z)} - 1\bigg| \to 0,
\end{align*}
uniformly for $0 < z \leq o(n^{1/6})$. 
\end{theorem}

It is worth mentioning that taking $\delta_{n} \asymp n^{-1/6}$ yields the wideest possible range $0 < z \leq o(n^{1/6})$ for the relative error in~\eqref{eq_moderate_deviation_mu_one_sample} to vanish, which is also optimal for the Cram\'{e}r-type moderate deviation results~\citep{Petrov1975, Fan2007, Liu2014, Zhou2018, Fan2019, Chen2020, Fang2020}. Next, we proceed to estimate the location shift parameter between two distributions via the HL estimator.



\subsection{Two-sample Problem}
A variety of applications use two-sample location shift estimation and inference, such as testing gene differences, quantifying treatment effects, and detecting change points. Accordingly, this section examines the two-sample estimation and inference of the population location shift parameter.

Let $Y_{j} = \theta^{\circ} + \varepsilon_{j}$, $j \in [m]$, be another sample of i.i.d.~real-valued random variables independent of $\{X_{1}, \ldots, X_{n}\}$, and we aim at constructing confidence interval for $\Theta = \theta - \theta^{\circ}$. Throughout this section, it is assumed that
\begin{align}
\label{ass:two}
    \Theta =  \inf\{t \in \mathbb{R} : \mathcal{U}(t) \geq 1/2\}, \qquad \cU(t) = \mathbb{P}(X_{1} - Y_{1} \leq t).
\end{align}
The existing literature on HL estimators mainly deals with the case where $\xi_{1}$ and $\varepsilon_{1}$ are identically distributed~\citep{Hodges1963, Lehmann1963, bauer1972constructing, rosenkranz2010note}. In contrast, it should be noted that the assumption imposed in~\eqref{ass:two} is satisfied as long as median$(\varepsilon_{1} - \xi_{1}) = 0,$ which is much more general than the identical distribution.  
In the sequel, following~\citet{Hodges1963}, the two sample HL estimator for $\Theta$ is given by 
\begin{align}
\label{eq_HL_estimator_two}
    \hat{\Theta} = \mathrm{median}\{X_{i} - Y_{j} : i \in [n], j \in [m]\}. 
\end{align}
Before proceeding, we present the following assumption on the relative sample sizes of the involved random samples. 
\begin{assumption}
\label{Assumption_two_sample_size}
There exists a positive constant $\bar{\eta} < 1$ such that $\bar{\eta} \leq (n/m) \leq 1/\bar{\eta}$. 
\end{assumption}

Assumption~\ref{Assumption_two_sample_size} is a natural condition which ensures the sample sizes to be comparable. Such a requirement is commonly imposed for two sample estimation and inference~\citep{Bai1996, Chen2010, Li2012, Chang2017, Zhang2020}. In what follows, we write $N = nm/(n + m)$ for simplicity. The sub-Gaussian-type deviation inequality and the non-asymptotic Bahadur representation of the two-sample HL estimator $\hat\Theta$ are established in the subsequent theorem.

\begin{theorem}
\label{Theorem_quantile_consistency_two}
Assume that there exist positive constants $\bar{c}_{0}$ and $\bar{\kappa}_{0}$ such that $\inf_{|\delta| \leq \bar{c}_{0}} \cU'(\Theta + \delta) \geq \bar{\kappa}_{0} > 0$. Then for any $z > 0$, we have 
\begin{align*}
    \PP(|\hat{\Theta} - \Theta| > z) \leq 4 \exp\{-2(n \wedge m)\bar{\kappa}_{0}^{2}(z\wedge \bar{c}_{0})^{2}\}.
\end{align*}
Furthermore, assume that $\sup_{t \in \mathbb{R}} |\mathcal{U}'(t)| < \infty$ and there exist positive constants $\bar{c}_{1}$ and $\bar{\kappa}_{1}$ such that $\sup_{|\delta| \leq \bar{c}_{1}} |\cU''(\Theta + \delta)| \leq \bar{\kappa}_{1}$. Then, under Assumption~\ref{Assumption_two_sample_size}, for any $0 < z = o(N)$, we have 
\begin{align*}
    \mathbb{P}\bigg\{\bigg|\hat{\Theta} - \Theta - \frac{1}{\mathcal{U}'(\Theta)} \bigg\{\frac{1}{n}\sum_{i = 1}^{n} G(\xi_{i}) - \frac{1}{m} \sum_{j = 1}^{m} F(\varepsilon_{j})\bigg\}\bigg| > \frac{C_{1} (z \vee 1)}{N}\bigg\} \leq C_{2} \exp(-z), 
\end{align*}
where $G(t) = \mathbb{P}(\varepsilon_{1} \leq t)$ stands for the cumulative distribution function of $\varepsilon_{1}$ and $C_{1}, C_{2} < \infty$ are positive constants depending only on $\bar{c}_{0}, \bar{\kappa}_{0}, \bar{c}_{1}, \bar{\kappa}_{1}, \bar{\eta}$ and $\|\mathcal{U}'\|_{\infty}$.  
\end{theorem}



Theorem \ref{Theorem_quantile_consistency_two} presents the non-asymptotic approximation of the HL estimator $\hat\Theta$, where the approximator also enjoys sub-Gaussian tails without posing any constraints on the moments of $\xi_{1}$ and $\varepsilon_{1}$. Equipped with this, we establish the Berry-Esseen bound and Cram\'{e}r type moderate deviation of $\hat\Theta$, respectively, in the following theorem. Before proceeding, we define the asymptotic variance of $\sqrt{N}(\hat{\Theta} - \Theta)$ to be 
\begin{align*}
    \tilde{\sigma}_{\Theta}^{2} = \frac{1}{\{\mathcal{U}'(\Theta)\}^{2}}\bigg(\frac{n}{n + m} \Var\{F(\varepsilon_{1})\} + \frac{m}{n + m} \Var\{G(\xi_{1})\}\bigg). 
\end{align*}

\begin{theorem}
\label{Theorem_Berry_Esseen_two}
Under the conditions of Theorem~\ref{Theorem_quantile_consistency_two}, we have
\begin{align*}
    \sup_{z \in \mathbb{R}} \bigg|\mathbb{P}\bigg\{\frac{\sqrt{N}(\hat{\Theta} - \Theta)}{\tilde{\sigma}_{\Theta}} \leq z\bigg\} - \Phi(z)\bigg| \leq \frac{C_{1} \log N}{\sqrt{N}}, 
\end{align*}
where $C_{1} < \infty$ is a positive constant independent of $N$. Moreover, let $\{\delta_{N}\}_{N \geq 1}$ be a sequence of positive constants satisfying $\sqrt{N}\delta_{N} \to \infty$. Then, we further achieve
\begin{align*}
    \bigg|\frac{\mathbb{P}\{\sqrt{N}(\hat{\Theta} - \Theta)/\tilde{\sigma}_{\Theta} \geq z\}}{1 - \Phi(z)} - 1\bigg| \leq C_{2} \bigg\{\frac{1 + z^{3}}{\sqrt{N}} + (1 + z)\delta_{N} + 2(z\vee 1) \sqrt{2\pi} \exp\bigg(\frac{z^{2}}{2} - \sqrt{N}\delta_{N}\bigg)\bigg\}, 
\end{align*}
uniformly for $0 < z \leq o(\delta_{N}^{-1} \wedge N^{1/4}\sqrt{\delta_{N}})$, where $C_{2} < \infty$ is a positive constant independent of $z$ and $N$. In particular, when $\delta_{N} \asymp N^{-1/6}$, we have 
\begin{align*}
    \bigg|\frac{\mathbb{P}\{\sqrt{N}(\hat{\Theta} - \Theta)/\tilde{\sigma}_{\Theta} \geq z\}}{1 - \Phi(z)} - 1\bigg| \to 0, 
\end{align*}
uniformly for $0 < z \leq o(N^{1/6})$. 
\end{theorem}


One observes that the asymptotic distributions of $\hat\theta$ and $\hat\Theta$ involve many unknown quantities such as density functions and population parameters $\theta$ and $\Theta$. In the following section, we utilize the bootstrap method to construct confidence intervals for the parameters of interest.

\section{Bootstrap Calibration}
\label{bootstrap}
In this section, we propose a weighted bootstrap method to construct confidence intervals for $\theta$ and $\Theta$, rather than directly estimating those involved unknown terms in asymptotic variances using the brute force methods. The reason is that the direct estimation approach always necessitates the additional selection of tuning parameters and imposes moment conditions. Additionally, the bootstrap calibration performs admirably with finite samples, particularly when the sample size is modest. Therefore, in the sections that follow, we outline the bootstrap procedures for both the one- and two-sample problems.

\subsection{Boostrap for One-sample Problem}
\label{bootstrap_one}
Recall that the one-sample HL estimator is given by $\hat{\theta} = \arg\min_{\nu \in \mathbb{R}} \sum_{i \neq j \in [n]} |X_{i} + X_{j} - 2\nu|$. Throughout this paper, we focus on the weighted bootstrap procedure in which the bootstrap estimate of $\hat{\theta}$ is defined by minimizing the randomly perturbed objective function. More specifically, let $\omega_{1}, \ldots, \omega_{n} \in \RR$ be i.i.d.~non-negative random variables with $\EE(\omega_{1}) = 1$ and $\Var(\omega_{1}) = 1$. Then the weighted bootstrap estimate of $\hat{\theta}$ is given by 
\begin{align*}
    \hat{\theta}^{\star} = \underset{\nu \in \RR}{\arg\min} \sum_{i \neq j \in [n]} \omega_{i} \omega_{j} |X_{i} + X_{j} - 2\nu|.
\end{align*}
A natural candidate of the bootstrap weight above would be $\omega_{i}$ sampled from a 2$\cdot$Bernoulli(0.5) distribution (the multiplication of 2  is to guarantee the previous normalization condition). In this case, the bootstrap estimator $\hat{\theta}^{\star}$ has the simple closed-form expression as follows, 
\begin{align}
\label{eq_one_sample_Bootstrap_estimator_median}
    \hat{\theta}^{\star} = \mathrm{median}\{(X_{i} + X_{j})/2 : i \neq j \in \mathcal{S}\},
\end{align}
which is the same as the sub-sampled HL estimator computed based on the dataset $\{X_{i} : i \in \mathcal{S}\}$ where $\mathcal{S} = \{i \in [n] : \omega_{i} \neq 0\}$, and we concentrate on this type of bootstrap calibration procedure in what follows.

Let $B_{n} = \sum_{i \neq j \in [n]} \omega_{i} \omega_{j}$ denote the total number of~\emph{Walsh averages} in~\eqref{eq_one_sample_Bootstrap_estimator_median} and denote $V_{ij} = [\mathbb{I}\{\xi_{i} + \xi_{j} \leq 0\} - U_{n}(\theta)]/U'(\theta)$ for each $i \neq j \in [n]$. In the subsequent theorem, we establish the non-asymptotic Bahadur representation of the bootstrap estimator $\hat\theta^\star$ and approximated distribution of bootstrap samples.

\begin{theorem}
\label{Theorem_quantile_consistency_one_Bootstrap}
Under the conditions of Theorem~\ref{Theorem_quantile_consistency_one}, for any $\omega, z > 0$ such that $(\omega \vee z) = o(n)$, with probability at least $1 - C_{1} \exp(-\omega)$, we have  
\begin{align}
    \mathbb{P}^{\star} \bigg\{|\hat{\theta}^{\star} - \hat{\theta}| > C_{2}\bigg(\frac{\omega + \log n}{n} + \sqrt{z/n}\bigg)\bigg\} &\leq C_{3} \exp(- z),\cr
    \mathbb{P}^{\star} \bigg\{\bigg|\hat{\theta}^{\star} - \hat{\theta} - \frac{1}{B_{n}} \sum_{i \neq j \in [n]} (\omega_{i} \omega_{j} - 1) V_{ij}\bigg| > \frac{C_{4}(z + \omega + \log n)}{n}\bigg\} &\leq C_{5}\exp(- z), \label{eq_Bahadur_bootstrap_one_sample}
\end{align}
where $\PP^{\star}(\cdot) = \PP(\cdot | X_{1}, \ldots, X_{n})$ stands for the conditional probability and $C_{1}$--$C_{7}$ are positive constants depending only on $c_{0}, \kappa_{0}, c_{1}, \kappa_{1}$ and $\|f\|_{\infty}$. 
\end{theorem}

The non-asymptotic linear expansion in~\eqref{eq_Bahadur_bootstrap_one_sample} enables us to derive the asymptotic normality of the bootstrap estimator $\hat{\theta}^{\star}$. Combined with the Berry-Esseen bound in Theorem~\ref{Theorem_Berry-Esseen_one}, we further establish a non-asymptotic upper bound on the Kolmogorov distance between the distribution functions of $\hat{\theta} - \theta$ and $\hat{\theta}^{\star} - \hat{\theta}$. More specifically, with probability at least $1 - C \exp(-\omega)$, we have  
\begin{align}
\label{eq_CLT_bootstrap_one_sample}
    \sup_{z \in \RR} \left|\PP^{\star}\left(|\hat{\theta}^{\star} - \hat{\theta}| \leq z\right) - \PP\left(|\hat{\theta} - \theta| \leq z\right)\right|\leq \frac{C'(\omega + \log n)}{\sqrt{n}}, 
\end{align}
where $C$ and $C'$ are positive constants independent of $n$. Consequently, we are equipped to construct confidence interval for $\theta$ in a data-driven way. For any significance level $\alpha \in (0, 1)$, let 
\begin{align*}
    q_{1 - \alpha}^{\star} = \inf\left\{z \in \RR : \PP^{\star}\left(|\hat{\theta}^{\star} - \hat{\theta}| \leq z\right) \geq 1 - \alpha\right\}. 
\end{align*}
Then the $(1 - \alpha)\times 100\%$ confidence interval for $\theta$ is given by $\mathbb{CI}(\theta, 1 - \alpha) = \{\hat{\theta} - q_{1 - \alpha}^{\star}, \hat{\theta} + q_{1 - \alpha}^{\star}\}$.


\subsection{Bootstrap in Two-sample Problem}
\label{bootstrap_two}
This section is devoted to constructing confidence intervals for $\Theta$ for the two-sample problem. Let $\omega_{n + 1}, \ldots, \omega_{n + m} \in \RR$ be i.i.d.~2$\cdot$Bernoulli(0.5) random variables independent of $\{\omega_{1}, \ldots, \omega_{n}\}$. Following~\eqref{eq_one_sample_Bootstrap_estimator_median}, the bootstrap estimator for $\hat{\Theta}$ is defined by 
\begin{align}
\label{eq_Bootstrap_Estimate_Two_Sample_one_dimension}
    \hat{\Theta}^{\star} = \textrm{median}\{X_{i} - Y_{j} : i \in \mathcal{S}^{X}, j \in \mathcal{S}^{Y}\}, 
\end{align} 
where $\mathcal{S}^{X} = \{i \in [n] : \omega_{i} \neq 0\}$ and $\mathcal{S}^{Y} = \{j \in [m] : \omega_{j + n} \neq 0\}$. It is worth noting that $\hat{\Theta}^{\star}$ is equivalent to the sub-sampled HL estimator based on the two datasets $\{X_{i} : i \in \mathcal{S}^{X}\}$ and $\{Y_{j} : j \in \mathcal{S}^{Y}\}$. With these necessary tools at hands, the non-asymptotic Bahadur representation of the bootstrap estimator $\hat\Theta^\star$ and approximated distribution of bootstrap samples are developed in the following theorem.

\begin{theorem}
\label{Theorem_Bootstrap_nonasymptotic_two_sample}
Under the conditions of Theorem~\ref{Theorem_quantile_consistency_two}, for any $\omega, z > 0$ such that $(\omega \vee z) = o(N)$, with probability at least $1 - C_{1}\exp(-\omega)$, we have  
\begin{align*}
    \mathbb{P}^{\star}\bigg\{|\hat{\Theta}^{\star} - \hat{\Theta}| > C_{2}\bigg(\frac{\omega + \log N}{N} + \sqrt{z/N}\bigg)\bigg\} &\leq C_{3} \exp(- z),\cr
    \PP^{\star}\bigg\{\bigg|\hat{\Theta}^{\star} - \hat{\Theta} - \frac{1}{\mathcal{B}_{n}} \sum_{i = 1}^{n} \sum_{j = 1}^{m} (\omega_{i} \omega_{j + n} - 1) \mathcal{V}_{ij}\bigg| > \frac{C_{4}(z + \omega + \log N)}{N}\bigg\} &\leq C_{5} \exp(-z),
\end{align*}
where $\mathbb{P}^{\star}(\cdot) = \mathbb{P}(\cdot | X_{1}, \ldots, X_{n}, Y_{1}, \ldots, Y_{n + m})$ stands for the conditional probability, $\mathcal{B}_{n} = \sum_{i = 1}^{n} \sum_{j = 1}^{m} \omega_{i} \omega_{j + n}$ is the total number of pairwise differences in~\eqref{eq_Bootstrap_Estimate_Two_Sample_one_dimension} and $\mathcal{V}_{ij} = [\mathbb{I}\{\xi_{i}\leq \varepsilon_{j}\} - (nm)^{-1} \sum_{i = 1}^{n} \sum_{j = 1}^{m} \mathbb{I}\{\xi_{i} \leq \varepsilon_{j}\}]/\mathcal{U}'(\Theta)$ for each $i \in [n]$ and $j \in [m]$. In addition, $C_{1}$--$C_{5}$ are positive constants depending only on $c_{0}, \kappa_{0}, c_{1}, \kappa_{1}, \bar{\eta}$ and $\|\mathcal{U}'\|_{\infty}$.
\end{theorem}


We then obtain the Berry-Esseen bound and build confidence intervals for $\Theta$ based on the non-asymptotic expansion of the two-sample bootstrap estimator $\hat\Theta^\star$ using similar arguments in~\S\ref{bootstrap_one}.

\section{Large-scale Multiple Testing}
\label{large-scale-one}
With the advancement of technology, large-scale, high-dimensional data have been extensively collected over the past two decades in a variety of fields such as medicine, biology, genetics, earth science, and finance. In large-scale regimes, there are inevitably heavy-tailed noises and it is crucial to develop robust statistical inference procedures. Yet, existing research that infers location shifts via Huber-type estimates calls for variable-dependent tuning parameters and moment limitations~\citep{Fan2019, sun2020adaptive}. This type of technique is hard to apply efficiently and faithfully to large-scale inferences due to the choices of tuning values. Moment constraints also exclude a large number of heavy-tailed distributions. To remedy the issues, this section focuses on extending the HL estimation to high dimensions and developing \emph{tuning-free} and \emph{moment-free} high-dimensional multiple testing procedures.



\subsection{Large-Scale Testing for One-sample Problem}
\label{large_scale_one}
In this section, we investigate high-dimensional multiple testing using the HL estimator for one-sample data. Let $\bX_{i} = \btheta + \bxi_{i}$, $i \in [n]$, be i.i.d.~\emph{p}-dimensional random vectors, where $\btheta = (\theta_{1}, \ldots, \theta_{p})^{\top}$ is a~\emph{p}-dimensional vector of unknown parameters and $\bxi_{1}, \ldots, \bxi_{n} \in \mathbb{R}^{p}$ are i.i.d.~random vectors. With building blocks presented in the previous section, we first proceed to constructing simultaneous confidence intervals for $\btheta$ using Gaussian approximation and bootstrap calibrations.

\subsubsection{Gaussian Approximation}\label{ga_one}

The primary goal of this section is to construct simultaneous confidence intervals for $\btheta$. To this end, we develop a Gaussian approximation for the maximum deviation $\max_{\ell \in [p]}|\hat{\theta}_{\ell} - \theta_{\ell}|$ following the intuition of recently developed high dimensional distributional theory~\citep{CCK2017, CCKK2022}. More specifically, let $\bZ = (Z_{1}, \ldots, Z_{p})^{\top}$ be a~\emph{p}-dimensional centered Gaussian random vector with
\begin{align}
\label{eq_GA_covariance_Z}
    \Cov(Z_{k}, Z_{\ell}) = \frac{4\Cov\{F_{k}(-\xi_{1k}), F_{\ell}(-\xi_{1\ell})\}}{U_{k}'(\theta_{k})U_{\ell}'(\theta_{\ell})}, \enspace k, \ell \in [p],
\end{align}
where  $U_{\ell}(t) = \mathbb{P}(X_{1\ell} + X_{2\ell} \leq 2t)$ and $U_{\ell}'(t)$ stands for its derivative, and $F_{\ell}(t) = \PP(\xi_{1\ell} \leq t)$. We have the high dimensional Gaussian approximation in the following theorem.


\begin{theorem}
\label{Theorem_GA_one}
Assume that there exist positive constants $c_{0}, \kappa_{0}, c_{1}$ and $\kappa_{1}$ such that   
\begin{align*}
    \min_{\ell \in [p]}\inf_{|\delta| \leq c_{0}} U_{\ell}'(\theta_{\ell} + \delta) \geq \kappa_{0} \enspace \mathrm{and} \enspace \max_{\ell \in [p]}\sup_{|\delta| \leq c_{1}} |U_{\ell}''(\theta_{\ell} + \delta)| \leq \kappa_{1}.
\end{align*}
Then, we have 
\begin{align}
\label{eq_GA_one_sample}
    \sup_{z > 0} \left|\PP\left(\max_{\ell \in [p]}\sqrt{n}|\hat{\theta}_{\ell} - \theta_{\ell}| \leq z\right) - \PP\left(\max_{\ell \in [p]}|Z_{\ell}| \leq z\right)\right| \leq \frac{C\log^{5/4}(pn)}{n^{1/4}}.
\end{align}
\end{theorem}

We consider testing the global null hypotheses
\begin{align}
\label{eq_high_dimensional_one_sample_testing}
    H_{0} : \theta_{\ell} = 0 \mbox{ for all } \ell \in [p] \enspace \mathrm{versus} \enspace H_{1} : \theta_{\ell} \neq 0 \mbox{ for some } \ell \in [p].
\end{align}
Based on the marginal HL estimators $\{\hat{\theta}_{\ell}\}_{\ell \in [p]}$, one shall reject the null hypothesis $H_{0}$ in~\eqref{eq_high_dimensional_one_sample_testing} when $\max_{\ell \in [p]} |\hat{\theta}_{\ell}|$ exceeds certain threshold that depends on the distribution of $\max_{\ell \in [p]} |Z_{\ell}|$. However, in light of~\eqref{eq_GA_covariance_Z}, the distribution depends on the unknown distribution functions $F_{\ell}$. Therefore, to approximate the distribution of $\max_{\ell \in [p]} |Z_{\ell}|$, we also propose to use bootstrap procedure. In specific, in one-sample regime, recall that $\mathcal{S} = \{i \in [n] : \omega_{i} \neq 0\}$. For each $\ell \in [p]$, define the bootstrap estimate of $\theta_{\ell}$ as $\hat{\theta}_{\ell}^{\star} = \mathrm{median}\{(X_{i\ell} + X_{j\ell})/2 : i\neq j \in \mathcal{S}\}$. It is worth mentioning that these bootstrap estimators $\{\hat{\theta}_{\ell}^{\star}\}_{\ell \in [p]}$ can be efficiently computed in practice. For $\alpha \in (0, 1)$, let $Q_{1 - \alpha}^{\star} = \inf\{z \in \mathbb{R} : \mathbb{P}^{\star}(\max_{\ell \in [p]} |\hat{\theta}_{\ell}^{\star} - \hat{\theta}_{\ell}| \leq z) \geq 1 - \alpha\}$ denote the $(1 - \alpha)$th quantile of the bootstrap statistic $\max_{\ell \in [p]} |\hat{\theta}_{\ell}^{\star} - \hat{\theta}_{\ell}|$. With the help of bootstrap, we manage to estimate the quantiles of the approximated distribution efficiently and the corresponding results are presented in the following theorem.

\begin{theorem}
\label{Theorem_Bootstrap_consistency_high_dimensional_GA_one_sample}
Under the conditions of Theorem~\ref{Theorem_GA_one}, we have 
\begin{align*}
    \left|\PP\left(\max_{\ell \in [p]}|\hat{\theta}_{\ell} - \theta_{\ell}| > Q_{1 - \alpha}^{\star}\right) - \alpha\right| \leq \frac{C\log^{5/4}(pn)}{n^{1/4}}.
\end{align*}
\end{theorem}

Theorem~\ref{Theorem_Bootstrap_consistency_high_dimensional_GA_one_sample} reveals that the proposed bootstrap procedure can efficiently estimate the quantiles of the approximated distribution. This allows for the direct construction of simultaneous data-driven confidence intervals for $\btheta$. In addition, when the null-hypothesis of~\eqref{eq_high_dimensional_one_sample_testing} is rejected, it is essential to conduct multiple testing to identify significant individuals and control false discovery proportion (FDP). We next address this problem in the following section.

\subsubsection{Multiple Testing}
\label{FDP_one}
The goal of this section is to conduct multiple testing to identify statistically significant individuals with controlled false discovery proportions. Specifically, we consider simultaneously testing the hypotheses 
\begin{align*}
    H_{0, \ell} : \theta_{\ell} = 0 \enspace \mathrm{versus} \enspace H_{1, \ell} : \theta_{\ell} \neq 0, \enspace \mathrm{for} \enspace \ell \in [p].
\end{align*}
Let $\mathcal{H}_{0} = \{\ell \in [p] : \theta_{\ell} = 0\}$ denote the set of true null hypotheses with cardinality $|\mathcal{H}_{0}|$. For each $\ell \in [p]$, let $P_{\ell}$ denote the~\emph{p}-value for testing the individual hypothesis $H_{0, \ell}$. For any prescribed threshold $t \in (0, 1)$, we shall reject the null hypothesis $H_{0, \ell}$ whenever $P_{\ell} < t$. Then the false discovery proportion is defined by  
\begin{align} \label{eq4.4}
    \mathrm{FDP}(t) = V(t)/\max\{R(t), 1\}, 
\end{align}
where $V(t) = \sum_{\ell \in \mathcal{H}_{0}} \mathbb{I}\{P_{\ell} \leq t\}$ denotes the number of false discoveries and $R(t) = \sum_{\ell = 1}^{p} \mathbb{I}\{P_{\ell} \leq t\}$ is the number of total discoveries. Note that the denominator is observable but $V(t)$ is not.  When $|\cH_0|$ tends to infinity, $V(t) \approx t |\cH_0| \leq t p$.  This can be used to give an upper bound of FDP$(t)$.  In many applications, $|\cH_0| \approx p$, and~\citet{storey2002direct} gives an estimator for $|\mathcal{H}_{0}|$ and incorporates it into $\mathrm{FDP}(t)$ estimator.


It is worth noting that the~\emph{p}-values $\{P_{\ell}\}_{\ell \in [p]}$ are computed by constructing inferential test statistics, with pivotal limiting distributions, based on the normal distribution calibration~\citep{Fan2007}. As illustrated above, to construct a test statistic for each hypothesis with pivotal asymptotic distribution,  the asymptotic variance of the HL estimator always depends on the unknown components and the traditional quantile-based approach is not scalable in the ultra-high dimensional scenario. To remedy this issue, we leverage bootstrap to proceed the analysis. Specifically, let $\{\omega_{i\ell} : i \in [n], \ell \in [p]\}$ be i.i.d.~non-negative random variables generated in the same way with those in~\S\ref{bootstrap_one} and denote $\mathcal{S}_{\ell} = \{i \in [n] : \omega_{i\ell} \neq 0\}$ for each $\ell \in [p]$. Similar to~\eqref{eq_one_sample_Bootstrap_estimator_median}, the bootstrap estimate of $\hat{\theta}_{\ell}$ is defined by 
\begin{align*}
    \hat{\theta}_{\ell}^{\star} = \mathrm{median}\{(X_{i\ell} + X_{j\ell})/2 : i \neq j \in \mathcal{S}_{\ell}\}. 
\end{align*}
Consequently, our~\emph{p}-values are derived as $P_{\ell} = \mathbb{P}(|\hat{\theta}_{\ell}^{\star} - \hat{\theta}_{\ell}| > |\hat{\theta}_{\ell}| |\bX_{1}, \ldots, \bX_{n})$ and let $P_{(1)} \leq P_{(2)} \leq \ldots \leq P_{(p)}$ denote the ordered~\emph{p}-values. In order to choose $t$ properly to control the FDP, we adopt the distribution-free procedure proposed by~\citet{Benjamini1995}. Specifically, for any significance level $\alpha \in (0, 1)$, the data-dependent threshold is $t_{BH} = P_{(\ell_{BH})}$, where $\ell_{BH}$ is given by
\begin{align*}
    \ell_{BH} = \max\{\ell \in [p] : P_{(\ell)} \leq \alpha \ell/p\}. 
\end{align*}
Recall that an estimate of $\mathrm{FDP}(t)$ is $\hat{\mathrm{FDP}}(t) = pt/R(t)$, following the discussion after~\eqref{eq4.4}. Then a natural choice of threshold is 
$$
   \hat{t} = \sup\{t:  \hat{\mathrm{FDP}}(t) \leq \alpha \} = \sup\{t:  t \leq \alpha R(t) /  p \} = t_{BH}.
$$
This provides a simple explanation on the choice of $t_{BH}$.

In what follows, we assume that $|\mathcal{H}_{0}|/p \to \pi_{0} \in (0, 1]$. For each $k, \ell \in [p]$, we define the correlation measure as $\rho_{k\ell} = \mathrm{Corr}\{F_{k}(-\xi_{1k}), F_{\ell}(-\xi_{1\ell})\}$, and impose the assumption quantifying dependence between measurements below.


\begin{assumption}
\label{Assumption_FDP_dependence_one}
There exists a positive constant $0 < \rho < 1$  such that $\max_{k \neq \ell} |\rho_{k\ell}| \leq \rho$ and
\begin{align*}
    \max_{\ell \in [p]} \sum_{k = 1}^{p} \mathbb{I}\left\{|\rho_{k\ell}| > \frac{1}{(\log p)^{2 + \kappa}}\right\} = \mathcal{O}(p^{\phi}),
\end{align*}
for some $\kappa > 0$ and $\phi \in (0, (1-\rho)/(1+\rho))$.
\end{assumption}

Assumption~\ref{Assumption_FDP_dependence_one} requires that for every variable, the number of other variables, whose correlations with the given variable exceed certain threshold, does not grow too fast. It is worth noting that this is a commonly imposed condition in large-scale multiple testing problems~\citep{Liu2014}. Based on this assumption, we next summarize the theoretical results in the following~Theorem~\ref{Theorem_FDP_consistency}.

\begin{theorem}
\label{Theorem_FDP_consistency}
Assume that $\log p = o(n^{1/5})$ and 
\begin{align}
\label{eq_cond_varpi}
    \varpi_{p} := \left|\left\{\ell \in [p] : |\theta_{\ell}/\sigma_{\ell}| \geq \lambda_{0} \sqrt{(2\log p)/n}\right\}\right| \to \infty, 
\end{align}
for some $\lambda_{0} > 2$, where $\sigma_{\ell}^{2} = 4 \mathrm{Var}\{F_{\ell}(-\xi_{1\ell})\}/\{U_{\ell}'(\theta_{\ell})\}^{2}$ for each $\ell \in [p]$. Then, under Assumption~\ref{Assumption_FDP_dependence_one} and the conditions of Theorem~\ref{Theorem_GA_one}, we have 
\begin{align}
\label{eq_FDP_one_sample_convergence}
    \bigg|\frac{\mathrm{FDP}(t_{BH})}{|\mathcal{H}_{0}|/p} - \alpha\bigg| \overset{\mathbb{P}}{\rightarrow} 0. 
\end{align}
\end{theorem}

Theorem \ref{Theorem_FDP_consistency} develops theoretical guarantees for the consistency of FDP control procedure. To further support the derived results, we make the following several remarks.
\begin{remark}
Condition~\eqref{eq_cond_varpi} is nearly optimal for controlling the false discovery proportion. More specifically, as shown in Proposition 2.1 in~\citet{Liu2014}, if the number of alternative hypotheses is fixed instead of tending to infinity, the B-H approach fails to control the FDP at any level $0 < \beta < 1$ even if the true~\emph{p}-values for $\mathcal{H}_{0, \ell}$ are known.  
\qed
\end{remark}

\subsection{Large-Scale Two-Sample Tests}
\label{large_scale_two}
In this section, we study the large-scale two-sample testings. In specific, let $\bY_{j} = \btheta^{\circ} + \bvarepsilon_{j}, j \in [m]$, be another sample of i.i.d.~\emph{p}-dimensional random vectors independent of $\{\bX_{1}, \ldots, \bX_{n}\}$, where $\btheta^{\circ} = (\theta_{1}^{\circ}, \ldots, \theta_{p}^{\circ})^{\top} \in \mathbb{R}^{p}$. Let $\bTheta = \btheta - \btheta^{\circ} = (\Theta_{1}, \ldots, \Theta_{p})^{\top} \in \RR^{p}$ be the location shift parameter. Following~\eqref{eq_HL_estimator_two}, the HL estimator for $\Theta_{\ell}$ is given by $\hat{\Theta}_{\ell} = \mathrm{median}\{X_{i\ell} - Y_{j\ell} : i \in [n], j \in [m]\}$. A detailed global test for whether $\theta_{\ell} = \theta_{\ell}^{\circ}$ for all $\ell \in [p]$ is given in~\S\ref{subsection_two_sample_simultaneous_testing}.


We next conduct a simultaneously test on the hypotheses 
\begin{align*}
    H_{0, \ell} : \theta_{\ell} = \theta_{\ell}^{\circ} \enspace \mathrm{versus} \enspace H_{1, \ell} : \theta_{\ell} \neq \theta_{\ell}^{\circ}, \enspace \mathrm{for} \enspace \ell \in [p],
\end{align*}
where $\mathcal{H}_{0}^{\diamond} = \{\ell \in [p] : \theta_{\ell} = \theta_{\ell}^{\circ}\}$ denote the set of true null hypotheses and $|\mathcal{H}_{0}^{\diamond}| = \sum_{\ell = 1}^{p} \mathbb{I}\{\theta_{\ell} = \theta_{\ell}^{\circ}\}$ is its cardinality. Throughout this section, we assume that $|\mathcal{H}_{0}^{\diamond}|/p \to \pi_{0} \in (0, 1]$. For each $\ell \in [p]$, let $P_{\ell}^{\diamond}$ denote the~\emph{p}-value for testing whether $\theta_{\ell} = \theta_{\ell}^{\circ}$. For any prescribed threshold $t \in (0, 1)$, we shall reject the null hypothesis $\theta_{\ell} = \theta_{\ell}^{\circ}$ whenever $P_{\ell}^{\diamond} < t$. The primary goal is to control the following false discovery proposition $\mathrm{FDP}^{\diamond}(t)$, 
\begin{align*}
    \mathrm{FDP}^{\diamond}(t) = V^{\diamond}(t)/\max\{R^{\diamond}(t), 1\},
\end{align*}
by selecting a proper $t$, where $V^{\diamond}(t) = \sum_{\ell \in \mathcal{H}_{0}^{\diamond}} \mathbb{I}\{P_{\ell}^{\diamond} \leq t\}$ and $R^{\diamond}(t) = \sum_{\ell = 1}^{p} \mathbb{I}\{P_{\ell}^{\diamond} \leq t\}$.

To approximate the unknown involved asymptotic distributions, we let $\{\omega_{i\ell} : i \in [n + m], \ell \in [p]\}$ be i.i.d.~non-negative random variables generated in the same way with those in \S\ref{bootstrap_two} and denote $\mathcal{S}_{\ell}^{X} = \{i \in [n] : \omega_{i\ell} \neq 0\}$ and $\mathcal{S}_{\ell}^{Y} = \{j \in [(n+1):(m+n)] : \omega_{j\ell} \neq 0\}$ for each $\ell \in [p]$. Following~\eqref{eq_Bootstrap_Estimate_Two_Sample_one_dimension}, the bootstrap estimate for $\hat{\Theta}_{\ell}$ is defined by
\begin{align*}
    \hat{\Theta}_{\ell}^{\star} = \mathrm{median}\{X_{i\ell} - Y_{j\ell} : i \in \mathcal{S}_{\ell}^{X}, j \in \mathcal{S}_{\ell}^{Y}\}. 
\end{align*}
Then, for each $\ell \in [p]$, the~\emph{p}-value is given by $P_{\ell}^{\diamond} = \mathbb{P}(|\hat{\Theta}_{\ell}^{\star} - \hat{\Theta}_{\ell}| > |\hat{\Theta}_{\ell}| | \bX_{1}, \ldots, \bX_{n})$, and $P_{(1)}^{\diamond} \leq P_{(2)}^{\diamond} \leq \ldots \leq P_{(p)}^{\diamond}$ are the ordered~\emph{p}-values. Following~\citet{Benjamini1995}, the data-dependent threshold $t_{BH}^{\diamond} = P_{(\ell_{BH}^{\diamond})}^{\diamond}$, where $\ell_{BH}^{\diamond}$ is chosen as $\ell_{BH}^{\diamond} = \max\{\ell \in [p] : P_{(\ell)}^{\diamond} \leq \alpha\ell/p\}$.


With these necessary tools at hand, in the paragraph that follows, we present the required assumptions and the main theorem, respectively. In specific, Assumption~\ref{Assumption_FDP_dependence_two} provides the formal condition that quantifies the level of dependence, while Theorem~\ref{Theorem_FDP_consistency_2} presents the theoretical assurances for two-sample multiple testing.

\begin{assumption}
\label{Assumption_FDP_dependence_two}
There exist positive constants $0 < \rho^{\diamond} < 1$ such that $\max_{k \neq \ell} |\rho_{k\ell}^{\diamond}| \leq \rho^{\diamond}$, where 
 $\rho_{k\ell}^{\diamond} = \mathrm{Corr}(\bar{G}_{nk} - \bar{F}_{mk}, \bar{G}_{n\ell} - \bar{F}_{m\ell})$, with $\bar{G}_{n\ell} = n^{-1}\sum_{i = 1}^{n} G_{\ell}(\xi_{i\ell})$ and $\bar{F}_{m\ell} = m^{-1} \sum_{j = 1}^{m} F_{\ell}(\varepsilon_{j\ell})$ for each $\ell \in [p]$. 
Moreover, for some $\kappa^{\diamond} > 0$ and $0 < \phi_{\diamond} < (1 - \rho^{\diamond})/(1 + \rho^{\diamond})$, we have
\begin{align*}
    \max_{\ell \in [p]} \sum_{k = 1}^{p} \mathbb{I}\bigg\{|\mathrm{Corr}\{G_{k}(\xi_{1k}), G_{\ell}(\xi_{1\ell})\}| \vee |\mathrm{Corr}\{F_{k}(\varepsilon_{1k}), F_{\ell}(\varepsilon_{1\ell})\}| > \frac{1}{(\log p)^{2 + \kappa^{\diamond}}}\bigg\} = \mathcal{O}(p^{\phi_{\diamond}}). 
\end{align*}
\end{assumption}

\begin{theorem}
\label{Theorem_FDP_consistency_2}
Assume that $\log p = o(N^{1/5})$ and 
\begin{align}
\label{eq_cond_varpi_2}
    \varpi_{p}^{\diamond} := \left|\left\{\ell \in [p] : |(\theta_{\ell} - \theta_{\ell}^{\circ})/\tilde{\sigma}_{\ell}| \geq \lambda_{0} \sqrt{(2\log p)/N}\right\}\right| \to \infty, 
\end{align}
for some $\lambda_{0} > 2$, where $\tilde{\sigma}_{\ell} = N[\mathrm{Var}\{F_{\ell}(\varepsilon_{1\ell})\}/m + \mathrm{Var}\{G_{\ell}(\xi_{1\ell})\}/n]/\{\mathcal{U}_{\ell}'(\Theta_{\ell})\}^{2}$ for each $\ell \in [p]$. Then, under Assumption~\ref{Assumption_FDP_dependence_two} and the conditions of Theorem~\ref{Theorem_two_sample_GA}, we have 
\begin{align}
\label{eq_FDP_one_sample_convergence_2}
    \bigg|\frac{\mathrm{FDP}^{\diamond}(t_{BH}^{\diamond})}{|\mathcal{H}_{0}^{\diamond}|/p} - \alpha\bigg| \overset{\mathbb{P}}{\rightarrow} 0. 
\end{align}
\end{theorem}

\section{Numerical Studies}\label{sec:num}
In this section, we use simulation experiments to verify the theoretical findings in the paper. In specific, we validate the results on Gaussian approximation and FDP control in via experiments in \S\ref{num:gauss} and \S\ref{num:FDP}, respectively.
\subsection{Numerical Studies for Global Tests}\label{num:gauss}
We let $n=m=300,p=400$ and generate the variables $\{\Xb_i\}_{i=1}^n$ and $\{\Yb_j\}_{j=1}^m$ following the setting in \S\ref{large_scale_one} and \S\ref{large_scale_two}, respectively, where the involved random variables $\xi_i\in \RR^p, \varepsilon_j\in \RR^p,(i\in[n],j\in[m])$ follow several distributions described below. Cases 1-3 are devoted primarily to one-sample tests, whereas Cases 4-6 are for two-sample tests. Moreover, the notation $\textrm{diff}(F)$ denotes the distribution that is generated by the differences between two independent random variables drawn from distribution $F$.
\begin{itemize}
    \item Case 1: Scaled $t_3$ distribution with $\xi_{i,k}\sim 0.3\cdot t_3,(i,k)\in [n]\times[p].$
    \item Case 2: Mixture  of Pareto distribution with shape parameter $2$ and standard Gaussian distribution, namely, $\xi_{i,k}\sim \textrm{diff}(0.2\cdot \textrm{Pareto}(2)+0.8\cdot N(0,1)),(i,k)\in[n]\times[p].$
    \item Case 3: Mixture of Gaussian distributions, namely, $\xi_i\in \RR^{p} \sim 0.2\cdot N(0,10\Sigma)+0.8\cdot N(0,\Sigma),i\in[n],$ where $\Sigma_{c,d}=0.7^{|c-d|},(c,d)\in[p]\times[p].$
    \item Case 4: Gaussian distributions but with different covariance,  namely, $\xi_{i}\sim  N(0,1.5\Sigma)$ and $\varepsilon_{j}\sim N(0,\Sigma),(i,j)\in[n]\times [m],$ where the covariance matrix is the same with that in case 3.
    \item Case 5: Mixture of Gaussian distributions with $\xi_i\in \RR^{p} \sim 0.2\cdot N(0,10\Sigma)+0.8\cdot N(0,\Sigma),i\in[n],$ where the covariance matrix is the same with that in case 3, and mixture of Pareto distribution with shape parameter 2 and standard Gaussian distribution, namely, $\varepsilon_{j,k}\sim \textrm{diff}(0.2\cdot \textrm{Pareto}(2)+0.8\cdot N(0,1)),(j,k)\in[m]\times[p].$
    \item Case 6: Scaled $t_3$ distribution, where $\xi_{i,k}\sim 0.3\cdot t_3,(i,k)\in[n]\times [p]$ and $\varepsilon_{j,k}\sim 0.3\cdot t_3, (j,k)\in[n]\times [p].$  
\end{itemize}
As a first step, we validate the results of Gaussian approximation of HL estimator by conducting the tests in \eqref{eq_high_dimensional_one_sample_testing} and \eqref{global_test_two} with threshold $\alpha=0.05$, respectively. We set the first 50 entries of $\btheta\in\RR^p$ (or $\bTheta=\btheta-\btheta_0\in \RR^p$) given in \S\ref{large_scale_one} and \S\ref{large_scale_two} as $\mu$, and the other entries as 0, where $\mu$ increases from 0 to 0.25. When $\mu=0$, the null-hypothesis test holds; otherwise, the alternative holds.
The size or power of the tests in \eqref{eq_high_dimensional_one_sample_testing} and \eqref{global_test_two} are computed via averaged outcomes from 500 replications of the methods in \S\ref{ga_one} and \S\ref{subsection_two_sample_simultaneous_testing}, respectively. In addition, for every replication, we conduct the weighted bootstrap method 300 times to compute the critical value of the test.

In addition, under the same experimental settings given above, we also compare the performance of HL estimator with the sample mean estimator ($\frac{1}{\sqrt{n}}\sum_{i=1}^n\bX_i$ in one-sample case and $\sqrt{nm/(n+m)}(\frac{1}{n}\sum_{i=1}^{n}{\bX_i}-\frac{1}{m}\sum_{j=1}^m\bY_j)$ in two-sample test). In this scenario, the critical values for tests \eqref{eq_high_dimensional_one_sample_testing} and \eqref{global_test_two} via sample mean estimators are computed via 300 bootstrap samples from joint Gaussian distribution $N(\mathbf{0},\hat\Sigma)$, where $\hat\Sigma$ is the sample covariance matrix ($\hat\Sigma=\frac{1}{n-1}\sum_{i=1}^n(\bX_i-\bar{\bX})(\bX_i-\bar{\bX})^\top$ in one-sample test and $\hat\Sigma= \frac{1}{n+m-2}\sum_{i=1}^n(\bX_i-\bar{\bX})(\bX_i-\bar{\bX})^\top+\frac{1}{m+n-2}\sum_{j=1}^m(\bY_j-\bar{\bY})(\bY_j-\bar{\bY})^\top$ in two-sample test). 
The results are then summarized in Table \ref{tab1}. 

    \begin{table}[h]
	\begin{center}
		\def\arraystretch{1}
		\setlength\tabcolsep{5pt}
		\begin{tabular}{c|c||c c c||c c c}
			\toprule
		\multirow{2}{*}{Estimator}&	\multirow{2}{*}{$\mu$} & \multicolumn{3}{c||}{One-sample} & \multicolumn{3}{c}{Two-sample}  \\
			& & Case 1 & Case 2 & Case 3  & Case 4 & Case 5 & Case 6 \\
			\hline
		\multirow{6}{*}{HL	}	&$\mu=0$	&0.054
			 &0.048 &0.052 &0.040
			 &0.046 &0.045   \\
			&$\mu=0.05$	&0.880
			 &0.062 &0.060 &0.080
			 &0.084 &0.282 \\
			 &$\mu=0.10$	&1.000
			 &0.480 &0.464 &0.164
			 & 0.140 &1.000 \\
			 	&		$\mu=0.15$	&1.000
			 &0.982 &0.942 &0.322
			 &0.320 &1.000   \\
			&$\mu=0.20$	&1.000
			 &1.000 &1.000 &0.728
			 &0.688 &1.000 \\
			& $\mu=0.25$	&1.000
			 &1.000 &1.000 &0.944
			 &1.000 &1.000 \\
			 \hline
		\multirow{6}{*}{Mean}	 		&	$\mu=0$	&0.004
			 &0.000 &0.044 &0.052
			 &0.000 &0.000   \\
		&	$\mu=0.05$	&0.042
			 &0.000 &0.046 &0.122
			 &0.000 &0.000 \\
		&	 $\mu=0.10$	&0.540
			 &0.000 &0.478 &0.200
			 &0.000 &0.166 \\
		&	 			$\mu=0.15$	&0.818
			 &0.002 &0.940 &0.384
			 &0.002 &0.718 \\
		&	$\mu=0.20$	&1.000
			 &0.006 &1.000 &0.742
			 &0.008 &0.858 \\
		&	 $\mu=0.25$	&1.000
			 &0.010 &1.000 &0.968
			 &0.020 &0.900 \\
			\bottomrule
		\end{tabular}
	\end{center}
	\caption{Sizes and powers for testing the global null using Gaussian approximation via HL estimator and sample mean estimator, respectively. }
	\label{tab1}
\end{table}
We conclude from Table \ref{tab1} that, in terms of the HL estimator, for all scenarios, the sizes of the test are roughly $0.05$ when the null hypothesis is true ($\mu=0$). Therefore, this validates the results of the Gaussian approximation. On the other hand, when the alternative holds, the power of the tests via HL estimator rises quickly to $1$ as $\mu$ increases. This demonstrates the effectiveness of the HL test statistics. However, in terms of sample mean estimator, one observes that when the null holds, the size of the test is approximately $0$ for most scenarios, whereas, when alternative holds, the power is much less than that of HL estimator. Thus, this further confirms the efficiency and robustness of HL estimator.

\subsection{Numerical Studies for FDP Control}
\label{num:FDP}
We validate the theoretical findings for FDP control under multiple regimes and also compare the performance of HL estimator with student's $t$-statistics given in \cite{Liu2014}. In specific, we maintain most of the settings mentioned in \S\ref{num:gauss}. The only difference is that we let the first $50$ entries of $\btheta$ and $\bTheta$ be $\mu$, where $\mu\in\{0.5,0.3\},$ and the rest ones being $0$.  The target false discover proportion $\alpha$ is varied uniformly from $0.05$ to $0.25$ and the empirical FDP are computed via the procedures in \S\ref{FDP_one} and \S\ref{large_scale_two}, respectively. Meanwhile, for both of these two statistics, besides the FDP, the true positive proportion (TPP) of the test $\textrm{TPP}(t_{BH})=\sum_{\ell\in\cH_1}\II\{P_\ell\le t_{BH}\}/|\cH_1|$ is also computed. The numerical outcomes are summarized in Table \ref{tab2} and Table \ref{tab3} (corresponds to cases with $\mu=0.5$ and $\mu=0.3$, respectively). 
\begin{table}[h]
	\begin{center}
		\def\arraystretch{1}
		\setlength\tabcolsep{7pt}
	
		\begin{tabular}{c|c|c c c c c}
			\toprule
		Estimator	& $\alpha$ & $0.05$ & $0.10$  &$0.15$ &$0.20$ &$0.25$ \\
			\hline
		\multirow{6}{*}{HL	}	  &Case 1 &0.068 (1.000)  &0.110 (1.000)  &0.152 (1.000) & 0.212 (1.000) &0.256 (1.000) \\
			&Case 2  &0.062 (1.000) &0.116 (1.000)  &0.154 (1.000) &0.198 (1.000) &0.254 (1.000)\\
			&Case 3  &0.052 (1.000) &0.102 (1.000)  &0.150 (1.000) &0.188 (1.000) &0.248 (1.000)\\
			  &Case 4 &0.058 (1.000)  &0.122 (1.000) &0.170 (1.000) &0.238 (1.000) &0.273 (1.000) \\
			&Case 5  &0.064 (1.000) &0.103 (1.000)  &0.142 (1.000) &0.187 (1.000) &0.234 (1.000)\\
			&Case 6  &0.062 (1.000) &0.092 (1.000) &0.151 (1.000) &0.193 (1.000) &0.232 (1.000)\\
\hline
\hline
		\multirow{6}{*}{Student's $t$}	 &Case 1 &0.071 (1.000)  &0.109 (1.000)  &0.168 (1.000) & 0.228 (1.000) &0.276 (1.000) \\
			&Case 2  &0.033 (0.976) &0.075 (0.976) &0.118 (0.990) &0.161 (0.996) &0.217 (1.000)\\
			&Case 3  &0.042 (1.000) &0.087 (1.000)  &0.131 (1.000) &0.206 (1.000) &0.255 (1.000)\\
			 &Case 4 &0.063 (1.000)  &0.127 (1.000) &0.181 (1.000) &0.246 (1.000) &0.299 (1.000) \\
			&Case 5  &0.026 (0.962) &0.064 (0.962) &0.110 (0.988) &0.159 (0.994) &0.195 (1.000)\\
			&Case 6  &0.048 (1.000) &0.106 (1.000) &0.168 (1.000) &0.196 (1.000) &0.272 (1.000)\\
			\bottomrule
		\end{tabular}
	\end{center}
	\caption{Empirical FDP and TPP versus the nominal level $\alpha$ of the test via HL estimator and student's $t$-statistics when $\mu=0.5$. The numbers outside and inside the brackets are averaged empirical false discovery proportions (FDP) and averaged true positive proportions (TPP), respectively, from 50 replications of the experiments in \S\ref{FDP_one} and \S\ref{large_scale_two}.  For every replication, we conduct bootstrap 300 times for every dimension to compute the corresponding empirical $p$-values for HL estimator using the method in \S\ref{FDP_one} and \S\ref{large_scale_two}. In addition, the $p$-values for student's $t$-statistics are computed via quantiles of the standard Gaussian distribution. }
	\label{tab2}
\end{table}

\begin{table}[h]
	\begin{center}
		\def\arraystretch{1}
		\setlength\tabcolsep{7pt}
	
		\begin{tabular}{c|c|c c c c c}
			\toprule
		Estimator	& $\alpha$ & $0.05$ & $0.10$  &$0.15$ &$0.20$ &$0.25$ \\
			\hline
		\multirow{6}{*}{HL	}	  &Case 1 &0.048 (1.000)  &0.112 (1.000)  &0.154 (1.000) & 0.202 (1.000) &0.256 (1.000) \\
			&Case 2  &0.074 (1.000) &0.133 (1.000)  &0.178 (1.000) &0.222 (1.000) &0.260 (1.000)\\
			&Case 3  &0.054 (1.000) &0.104 (1.000)  &0.153 (1.000) &0.210 (1.000) &0.254 (1.000)\\
			  &Case 4 &0.006 (0.984)  &0.048 (0.998) &0.108 (1.000) &0.176 (1.000) &0.232 (1.000) \\
			&Case 5  &0.000 (0.992) &0.024 (0.992)  &0.092 (0.992) &0.166 (0.992) &0.200 (0.992)\\
			&Case 6  &0.055 (1.000) &0.095 (1.000) &0.138 (1.000) &0.184 (1.000) &0.254 (1.000)\\
\hline
\hline
		\multirow{6}{*}{Student's $t$}	 &Case 1 &0.044 (1.000)  &0.098 (1.000)  &0.166 (1.000) & 0.204 (1.000) &0.263 (1.000) \\
			&Case 2  &0.011 (0.910) &0.056 (0.910) &0.106 (0.930) &0.161 (0.968) &0.205 (0.984)\\
			&Case 3  &0.052 (1.000) &0.106 (1.000)  &0.150 (1.000) &0.198 (1.000) &0.259 (1.000)\\
			 &Case 4 &0.004 (1.000)  &0.046 (1.000) &0.110 (1.000) &0.158 (1.000) &0.210 (1.000) \\
			&Case 5  &0.000 (0.870) &0.000 (0.952) &0.056 (0.984) &0.100 (0.990) &0.170 (1.000)\\
			&Case 6  &0.046 (1.000) &0.106 (1.000) &0.140 (1.000) &0.180 (1.000) &0.262 (1.000)\\
			\bottomrule
		\end{tabular}
	\end{center}
	\caption{Empirical FDP and TPP versus the nominal level $\alpha$ of the test via HL estimator and student's $t$-statistics when $\mu=0.3$. The remain captions are the same with those in Table \ref{tab2}. }
	\label{tab3}
\end{table}

 Compared with the student's t-statistics, the empirical FDPs of HL estimator are closer to the theoretical thresholds in most cases and outcomes based on HL estimator have larger true positive proportion (TPP). This validates the theory of FDP control and confirms the benefit and effectiveness of using HL estimator when heavy-tailed error exists. In addition, we also further compare the performance of HL estimator with the student's t-statistics in both one-sample and two-sample tests when the noises follow $t_1$ distribution, where the HL estimator also performs much better. Interested readers are referred to  \S\ref{add_simu} for more details.


\section{Conclusion}\label{sec:conc}
In large-scale data analysis, conventional methods are ineffective since outliers and variables with heavy tails can easily corrupt the data. To resolve heavy-tailed contamination, existing techniques, such as Huberized mean, truncation, and median of means, always require additional tuning parameters and moment restrictions. Consequently, they cannot effectively be scaled to the high-dimensional applications with fidelity. Using the well-known Hodge-Lehmann estimator, the constraint on moment and tuning parameters can be removed. However, its non-asymptotic and large-scale properties have never been investigated. This paper fills this important gap by contributing a finite-sample analysis of the HL estimator, generalizing it to large-scale studies, and proposing tuning-free and moment-free high-dimensional testing methods.

There are various potential future directions that merit further studies. First, we permit mild measurement dependence while controlling the FDP in both the one- and two-sample regimes. However, in reality, high-dimensional data can  exhibit strong dependency such as those collected in the field of economics, finance,  genomics, and meteorology. To solve such strong dependence problems, factor-adjusted multiple testing via Huber-type estimation has been proposed \citep{Fan2019}. However,  using Huber-type loss will result in extra tuning parameters and moment constraints. As a result, factor adjustments can be incorporated into the large-scale HL estimation and testing procedures in order to develop tuning-free procedures that can be adapted to a strong dependence scenario. Second, in terms of computation, we need to compute the median of $\cO(n^2)$ pairs. However, in the one sample estimation regime, if we use the sub-sampling idea, one may compute the median of $n/2$ non-overlapping pair averages:
\begin{align*}
    \theta^\tau=\textrm{median}\left(\frac{X_{\tau(2i-1)}+X_{\tau(2i)}}{2}, i\in[n/2]\right)
\end{align*}
with a random permutation $\tau$ on $[n].$ By taking 5 or more permutations and the averages of these estimates, numerically, the averaged one approximates well the  Hodge-Lehmann estimator and only requires $\cO(n)$ samples \citep{fan2020comment}. Therefore, it is also interesting to derive non-asymptotic analysis for this estimator based on sub-sampling.

\newpage
\bibliographystyle{ims}
\bibliography{dynamic}

\begin{thebibliography}{77}
\expandafter\ifx\csname natexlab\endcsname\relax\def\natexlab#1{#1}\fi
\expandafter\ifx\csname url\endcsname\relax
  \def\url#1{\texttt{#1}}\fi
\expandafter\ifx\csname urlprefix\endcsname\relax\def\urlprefix{}\fi

\bibitem[{Arcones(1995)}]{Arcones1995}
\text{Arcones, M.~A.} (1995).
\newblock The asymptotic accuracy of the bootstrap of u-quantiles.
\newblock \textit{Ann. Statist.} 1802--1822.

\bibitem[{Arcones(1996)}]{Arcones1996}
\text{Arcones, M.~A.} (1996).
\newblock The bahadur-kiefer representation for u-quantiles.
\newblock \textit{Ann. Statist.}, \textbf{24} 1400--1422.

\bibitem[{Bai and Saranadasa(1996)}]{Bai1996}
\text{Bai, Z.} and \text{Saranadasa, H.} (1996).
\newblock Effect of high dimension: by an example of a two sample problem.
\newblock \textit{Statist. Sinica}, \textbf{6} 311--329.

\bibitem[{Bauer(1972)}]{bauer1972constructing}
\text{Bauer, D.~F.} (1972).
\newblock Constructing confidence sets using rank statistics.
\newblock \textit{J. Amer. Statist. Assoc.}, \textbf{67} 687--690.

\bibitem[{Belloni and Chernozhukov(2011)}]{Belloni2011}
\text{Belloni, A.} and \text{Chernozhukov, V.} (2011).
\newblock {$\ell_1$}-penalized quantile regression in high-dimensional sparse
  models.
\newblock \textit{Ann. Statist.}, \textbf{39} 82--130.

\bibitem[{Benjamini and Hochberg(1995)}]{Benjamini1995}
\text{Benjamini, Y.} and \text{Hochberg, Y.} (1995).
\newblock Controlling the false discovery rate: a practical and powerful
  approach to multiple testing.
\newblock \textit{J. Roy. Statist. Soc. Ser. B}, \textbf{57} 289--300.

\bibitem[{Bentkus and Dzindzalieta(2015)}]{bentkus2015tight}
\text{Bentkus, V.~K.} and \text{Dzindzalieta, D.} (2015).
\newblock A tight gaussian bound for weighted sums of rademacher random
  variables.
\newblock \textit{Bernoulli}, \textbf{21} 1231--1237.

\bibitem[{Blanchard and Roquain(2009)}]{blanchard2009adaptive}
\text{Blanchard, G.} and \text{Roquain, {\'E}.} (2009).
\newblock Adaptive false discovery rate control under independence and
  dependence.
\newblock \textit{Journal of Machine Learning Research}, \textbf{10}.

\bibitem[{Brownlees et~al.(2015)Brownlees, Joly and
  Lugosi}]{brownlees2015empirical}
\text{Brownlees, C.}, \text{Joly, E.} and \text{Lugosi, G.} (2015).
\newblock Empirical risk minimization for heavy-tailed losses.
\newblock \textit{The Annals of Statistics}, \textbf{43} 2507--2536.

\bibitem[{B{\"u}hlmann and Van De~Geer(2011)}]{buHLmann2011statistics}
\text{B{\"u}hlmann, P.} and \text{Van De~Geer, S.} (2011).
\newblock \textit{Statistics for high-dimensional data: methods, theory and
  applications}.
\newblock Springer Science \& Business Media.

\bibitem[{Catoni(2012)}]{catoni2012challenging}
\text{Catoni, O.} (2012).
\newblock Challenging the empirical mean and empirical variance: a deviation
  study.
\newblock In \textit{Annales de l'IHP Probabilit{\'e}s et statistiques},
  vol.~48.

\bibitem[{Chang et~al.(2017)Chang, Zheng, Zhou and Zhou}]{Chang2017}
\text{Chang, J.}, \text{Zheng, C.}, \text{Zhou, W.-X.} and \text{Zhou, W.}
  (2017).
\newblock Simulation-based hypothesis testing of high dimensional means under
  covariance heterogeneity.
\newblock \textit{Biometrics}, \textbf{73} 1300--1310.

\bibitem[{Chen and Qin(2010)}]{Chen2010}
\text{Chen, S.~X.} and \text{Qin, Y.-L.} (2010).
\newblock A two-sample test for high-dimensional data with applications to
  gene-set testing.
\newblock \textit{Ann. Statist.}, \textbf{38} 808--835.

\bibitem[{Chen and Zhou(2020)}]{Chen2020}
\text{Chen, X.} and \text{Zhou, W.-X.} (2020).
\newblock Robust inference via multiplier bootstrap.
\newblock \textit{Ann. Statist.}, \textbf{48} 1665--1691.

\bibitem[{Chen et~al.(2021)Chen, Chi, Fan, Ma et~al.}]{chen2021spectral}
\text{Chen, Y.}, \text{Chi, Y.}, \text{Fan, J.}, \text{Ma, C.} \text{et~al.}
  (2021).
\newblock Spectral methods for data science: A statistical perspective.
\newblock \textit{Foundations and Trends{\textregistered} in Machine Learning},
  \textbf{14} 566--806.

\bibitem[{Chernozhukov et~al.(2017)Chernozhukov, Chetverikov and
  Kato}]{CCK2017}
\text{Chernozhukov, V.}, \text{Chetverikov, D.} and \text{Kato, K.} (2017).
\newblock Central limit theorems and bootstrap in high dimensions.
\newblock \textit{Ann. Probab.}, \textbf{45} 2309--2352.

\bibitem[{Chernozhuokov et~al.(2022)Chernozhuokov, Chetverikov, Kato and
  Koike}]{CCKK2022}
\text{Chernozhuokov, V.}, \text{Chetverikov, D.}, \text{Kato, K.} and
  \text{Koike, Y.} (2022).
\newblock Improved central limit theorem and bootstrap approximations in high
  dimensions.
\newblock \textit{Ann. Statist.}, \textbf{50} 2562--2586.

\bibitem[{Chi(2007)}]{chi2007performance}
\text{Chi, Z.} (2007).
\newblock On the performance of fdr control: constraints and a partial
  solution.
\newblock \textit{The Annals of Statistics}, \textbf{35} 1409--1431.

\bibitem[{Cont(2001)}]{cont2001empirical}
\text{Cont, R.} (2001).
\newblock Empirical properties of asset returns: stylized facts and statistical
  issues.
\newblock \textit{Quantitative finance}, \textbf{1} 223.

\bibitem[{Eklund et~al.(2016)Eklund, Nichols and Knutsson}]{eklund2016cluster}
\text{Eklund, A.}, \text{Nichols, T.~E.} and \text{Knutsson, H.} (2016).
\newblock Cluster failure: Why fmri inferences for spatial extent have inflated
  false-positive rates.
\newblock \textit{Proceedings of the national academy of sciences},
  \textbf{113} 7900--7905.

\bibitem[{Fan et~al.(2014)Fan, Fan and Barut}]{fan2014adaptive}
\text{Fan, J.}, \text{Fan, Y.} and \text{Barut, E.} (2014).
\newblock Adaptive robust variable selection.
\newblock \textit{Ann. Statist.}, \textbf{42} 324--351.

\bibitem[{Fan et~al.(2022{\natexlab{a}})Fan, Gu and Zhou}]{fan2022noise}
\text{Fan, J.}, \text{Gu, Y.} and \text{Zhou, W.-X.} (2022{\natexlab{a}}).
\newblock How do noise tails impact on deep relu networks?
\newblock \textit{arXiv preprint arXiv:2203.10418}.

\bibitem[{Fan et~al.(2007)Fan, Hall and Yao}]{Fan2007}
\text{Fan, J.}, \text{Hall, P.} and \text{Yao, Q.} (2007).
\newblock To how many simultaneous hypothesis tests can normal, {S}tudent's
  {$t$} or bootstrap calibration be applied?
\newblock \textit{J. Amer. Statist. Assoc.}, \textbf{102} 1282--1288.

\bibitem[{Fan et~al.(2019)Fan, Ke, Sun and Zhou}]{Fan2019}
\text{Fan, J.}, \text{Ke, Y.}, \text{Sun, Q.} and \text{Zhou, W.-X.} (2019).
\newblock Farm{T}est: factor-adjusted robust multiple testing with approximate
  false discovery control.
\newblock \textit{J. Amer. Statist. Assoc.}, \textbf{114} 1880--1893.

\bibitem[{Fan et~al.(2017)Fan, Li and Wang}]{fan2017estimation}
\text{Fan, J.}, \text{Li, Q.} and \text{Wang, Y.} (2017).
\newblock Estimation of high dimensional mean regression in the absence of
  symmetry and light tail assumptions.
\newblock \textit{J. R. Stat. Soc. Ser. B. Stat. Methodol.}, \textbf{79}
  247--265.

\bibitem[{Fan et~al.(2020{\natexlab{a}})Fan, Li, Zhang and
  Zou}]{fan2020statistical}
\text{Fan, J.}, \text{Li, R.}, \text{Zhang, C.-H.} and \text{Zou, H.}
  (2020{\natexlab{a}}).
\newblock \textit{Statistical foundations of data science}.
\newblock Chapman and Hall/CRC.

\bibitem[{Fan et~al.(2022{\natexlab{b}})Fan, Lou and Yu}]{fan2022latent}
\text{Fan, J.}, \text{Lou, Z.} and \text{Yu, M.} (2022{\natexlab{b}}).
\newblock Are latent factor regression and sparse regression adequate?
\newblock \textit{arXiv preprint arXiv:2203.01219}.

\bibitem[{Fan et~al.(2020{\natexlab{b}})Fan, Ma and Wang}]{fan2020comment}
\text{Fan, J.}, \text{Ma, C.} and \text{Wang, K.} (2020{\natexlab{b}}).
\newblock Comment on “a tuning-free robust and efficient approach to
  high-dimensional regression”.
\newblock \textit{Journal of the American Statistical Association},
  \textbf{115} 1720--1725.

\bibitem[{Fan et~al.(2021)Fan, Wang and Zhu}]{fan2021shrinkage}
\text{Fan, J.}, \text{Wang, W.} and \text{Zhu, Z.} (2021).
\newblock A shrinkage principle for heavy-tailed data: high-dimensional robust
  low-rank matrix recovery.
\newblock \textit{Ann. Statist.}, \textbf{49} 1239--1266.

\bibitem[{Fan et~al.(2022{\natexlab{c}})Fan, Yang and
  Yu}]{fan2022understanding}
\text{Fan, J.}, \text{Yang, Z.} and \text{Yu, M.} (2022{\natexlab{c}}).
\newblock Understanding implicit regularization in over-parameterized single
  index model.
\newblock \textit{Journal of the American Statistical Association} 1--14.

\bibitem[{Fan and Yao(2017)}]{fan2017elements}
\text{Fan, J.} and \text{Yao, Q.} (2017).
\newblock \textit{The elements of financial econometrics}.
\newblock Cambridge University Press.

\bibitem[{Fang et~al.(2020)Fang, Luo and Shao}]{Fang2020}
\text{Fang, X.}, \text{Luo, L.} and \text{Shao, Q.-M.} (2020).
\newblock A refined {C}ram\'{e}r-type moderate deviation for sums of local
  statistics.
\newblock \textit{Bernoulli}, \textbf{26} 2319--2352.

\bibitem[{Ferreira and Zwinderman(2006)}]{ferreira2006benjamini}
\text{Ferreira, J.~A.} and \text{Zwinderman, A.~H.} (2006).
\newblock On the {B}enjamini-{H}ochberg method.
\newblock \textit{Ann. Statist.}, \textbf{34} 1827--1849.

\bibitem[{Finotello and Di~Camillo(2015)}]{finotello2015measuring}
\text{Finotello, F.} and \text{Di~Camillo, B.} (2015).
\newblock Measuring differential gene expression with rna-seq: challenges and
  strategies for data analysis.
\newblock \textit{Briefings in functional genomics}, \textbf{14} 130--142.

\bibitem[{Genovese and Wasserman(2004)}]{genovese2004stochastic}
\text{Genovese, C.} and \text{Wasserman, L.} (2004).
\newblock A stochastic process approach to false discovery control.
\newblock \textit{The annals of statistics}, \textbf{32} 1035--1061.

\bibitem[{Gin{\'e} et~al.(2000)Gin{\'e}, Lata{\l}a and
  Zinn}]{gine2000exponential}
\text{Gin{\'e}, E.}, \text{Lata{\l}a, R.} and \text{Zinn, J.} (2000).
\newblock Exponential and moment inequalities for u-statistics.
\newblock In \textit{High Dimensional Probability II}. Springer, 13--38.

\bibitem[{Goldstein et~al.(2018)Goldstein, Minsker and
  Wei}]{goldstein2018structured}
\text{Goldstein, L.}, \text{Minsker, S.} and \text{Wei, X.} (2018).
\newblock Structured signal recovery from non-linear and heavy-tailed
  measurements.
\newblock \textit{IEEE Transactions on Information Theory}, \textbf{64}
  5513--5530.

\bibitem[{Gupta et~al.(2014)Gupta, Ellis, Ashar, Moes, Bader, Zhan, West and
  Arking}]{gupta2014transcriptome}
\text{Gupta, S.}, \text{Ellis, S.~E.}, \text{Ashar, F.~N.}, \text{Moes, A.},
  \text{Bader, J.~S.}, \text{Zhan, J.}, \text{West, A.~B.} and \text{Arking,
  D.~E.} (2014).
\newblock Transcriptome analysis reveals dysregulation of innate immune
  response genes and neuronal activity-dependent genes in autism.
\newblock \textit{Nature communications}, \textbf{5} 1--8.

\bibitem[{Hastie et~al.(2009)Hastie, Tibshirani, Friedman and
  Friedman}]{hastie2009elements}
\text{Hastie, T.}, \text{Tibshirani, R.}, \text{Friedman, J.~H.} and
  \text{Friedman, J.~H.} (2009).
\newblock \textit{The elements of statistical learning: data mining, inference,
  and prediction}, vol.~2.
\newblock Springer.

\bibitem[{Hastie et~al.(2015)Hastie, Tibshirani and
  Wainwright}]{hastie2015statistical}
\text{Hastie, T.}, \text{Tibshirani, R.} and \text{Wainwright, M.} (2015).
\newblock Statistical learning with sparsity.
\newblock \textit{Monographs on statistics and applied probability},
  \textbf{143} 143.

\bibitem[{He and Shao(1996)}]{he1996general}
\text{He, X.} and \text{Shao, Q.-M.} (1996).
\newblock A general bahadur representation of m-estimators and its application
  to linear regression with nonstochastic designs.
\newblock \textit{The Annals of Statistics}, \textbf{24} 2608--2630.

\bibitem[{He and Shao(2000)}]{he2000parameters}
\text{He, X.} and \text{Shao, Q.-M.} (2000).
\newblock On parameters of increasing dimensions.
\newblock \textit{Journal of multivariate analysis}, \textbf{73} 120--135.

\bibitem[{Hodges and Lehmann(1963)}]{Hodges1963}
\text{Hodges, J.~L., Jr.} and \text{Lehmann, E.~L.} (1963).
\newblock Estimates of location based on rank tests.
\newblock \textit{Ann. Math. Statist.}, \textbf{34} 598--611.

\bibitem[{H{\o}yland(1965)}]{Hoyland1965}
\text{H{\o}yland, A.} (1965).
\newblock Robustness of the {H}odges-{L}ehmann estimates for shift.
\newblock \textit{Ann. Math. Statist.}, \textbf{36} 174--197.

\bibitem[{Hsu and Sabato(2016)}]{hsu2016loss}
\text{Hsu, D.} and \text{Sabato, S.} (2016).
\newblock Loss minimization and parameter estimation with heavy tails.
\newblock \textit{The Journal of Machine Learning Research}, \textbf{17}
  543--582.

\bibitem[{Huber(1973)}]{huber1973robust}
\text{Huber, P.~J.} (1973).
\newblock Robust regression: asymptotics, conjectures and monte carlo.
\newblock \textit{The annals of statistics} 799--821.

\bibitem[{Koenker and Hallock(2001)}]{koenker2001quantile}
\text{Koenker, R.} and \text{Hallock, K.~F.} (2001).
\newblock Quantile regression.
\newblock \textit{Journal of economic perspectives}, \textbf{15} 143--156.

\bibitem[{Lehmann(1963)}]{Lehmann1963}
\text{Lehmann, E.~L.} (1963).
\newblock Nonparametric confidence intervals for a shift parameter.
\newblock \textit{Ann. Math. Statist.}, \textbf{34} 1507--1512.

\bibitem[{Li and Chen(2012)}]{Li2012}
\text{Li, J.} and \text{Chen, S.~X.} (2012).
\newblock Two sample tests for high-dimensional covariance matrices.
\newblock \textit{Ann. Statist.}, \textbf{40} 908--940.

\bibitem[{Li and Tibshirani(2013)}]{li2013finding}
\text{Li, J.} and \text{Tibshirani, R.} (2013).
\newblock Finding consistent patterns: a nonparametric approach for identifying
  differential expression in {RNA}-{S}eq data.
\newblock \textit{Stat. Methods Med. Res.}, \textbf{22} 519--536.

\bibitem[{Li et~al.(2012)Li, Witten, Johnstone and
  Tibshirani}]{li2012normalization}
\text{Li, J.}, \text{Witten, D.~M.}, \text{Johnstone, I.~M.} and
  \text{Tibshirani, R.} (2012).
\newblock Normalization, testing, and false discovery rate estimation for
  rna-sequencing data.
\newblock \textit{Biostatistics}, \textbf{13} 523--538.

\bibitem[{Liu and Shao(2014)}]{Liu2014}
\text{Liu, W.} and \text{Shao, Q.-M.} (2014).
\newblock Phase transition and regularized bootstrap in large-scale {$t$}-tests
  with false discovery rate control.
\newblock \textit{Ann. Statist.}, \textbf{42} 2003--2025.

\bibitem[{Loh(2017)}]{loh2017statistical}
\text{Loh, P.-L.} (2017).
\newblock Statistical consistency and asymptotic normality for high-dimensional
  robust $ m $-estimators.
\newblock \textit{The Annals of Statistics}, \textbf{45} 866--896.

\bibitem[{Mammen(1989)}]{mammen1989asymptotics}
\text{Mammen, E.} (1989).
\newblock Asymptotics with increasing dimension for robust regression with
  applications to the bootstrap.
\newblock \textit{The Annals of Statistics} 382--400.

\bibitem[{Minsker(2015)}]{minsker2015geometric}
\text{Minsker, S.} (2015).
\newblock Geometric median and robust estimation in banach spaces.
\newblock \textit{Bernoulli}, \textbf{21} 2308--2335.

\bibitem[{Minsker(2018)}]{minsker2018sub}
\text{Minsker, S.} (2018).
\newblock Sub-gaussian estimators of the mean of a random matrix with
  heavy-tailed entries.
\newblock \textit{The Annals of Statistics}, \textbf{46} 2871--2903.

\bibitem[{Nagalakshmi et~al.(2008)Nagalakshmi, Wang, Waern, Shou, Raha,
  Gerstein and Snyder}]{nagalakshmi2008transcriptional}
\text{Nagalakshmi, U.}, \text{Wang, Z.}, \text{Waern, K.}, \text{Shou, C.},
  \text{Raha, D.}, \text{Gerstein, M.} and \text{Snyder, M.} (2008).
\newblock The transcriptional landscape of the yeast genome defined by rna
  sequencing.
\newblock \textit{Science}, \textbf{320} 1344--1349.

\bibitem[{Nemirovsky and Yudin(1983)}]{nemirovskij1983problem}
\text{Nemirovsky, A.~S.} and \text{Yudin, D. B.~a.} (1983).
\newblock \textit{Problem complexity and method efficiency in optimization}.
\newblock Wiley-Interscience Series in Discrete Mathematics, John Wiley \&
  Sons, Inc., New York.

\bibitem[{Petrov(1975)}]{Petrov1975}
\text{Petrov, V.~V.} (1975).
\newblock \textit{Sums of independent random variables}.
\newblock Ergebnisse der Mathematik und ihrer Grenzgebiete, Band 82,
  Springer-Verlag, New York-Heidelberg.
\newblock Translated from the Russian by A. A. Brown.

\bibitem[{Rosenkranz(2010)}]{rosenkranz2010note}
\text{Rosenkranz, G.~K.} (2010).
\newblock A note on the hodges--lehmann estimator.
\newblock \textit{Pharmaceutical statistics}, \textbf{9} 162--167.

\bibitem[{Shendure and Ji(2008)}]{shendure2008next}
\text{Shendure, J.} and \text{Ji, H.} (2008).
\newblock Next-generation dna sequencing.
\newblock \textit{Nature biotechnology}, \textbf{26} 1135--1145.

\bibitem[{Stock and Watson(2002)}]{stock2002macroeconomic}
\text{Stock, J.~H.} and \text{Watson, M.~W.} (2002).
\newblock Macroeconomic forecasting using diffusion indexes.
\newblock \textit{Journal of Business \& Economic Statistics}, \textbf{20}
  147--162.

\bibitem[{Storey(2002)}]{storey2002direct}
\text{Storey, J.~D.} (2002).
\newblock A direct approach to false discovery rates.
\newblock \textit{J. R. Stat. Soc. Ser. B Stat. Methodol.}, \textbf{64}
  479--498.

\bibitem[{Storey(2003)}]{storey2003positive}
\text{Storey, J.~D.} (2003).
\newblock The positive false discovery rate: a {B}ayesian interpretation and
  the {$q$}-value.
\newblock \textit{Ann. Statist.}, \textbf{31} 2013--2035.

\bibitem[{Storey et~al.(2004)Storey, Taylor and Siegmund}]{Storey2004}
\text{Storey, J.~D.}, \text{Taylor, J.~E.} and \text{Siegmund, D.} (2004).
\newblock Strong control, conservative point estimation and simultaneous
  conservative consistency of false discovery rates: a unified approach.
\newblock \textit{J. R. Stat. Soc. Ser. B Stat. Methodol.}, \textbf{66}
  187--205.

\bibitem[{Sun et~al.(2020)Sun, Zhou and Fan}]{sun2020adaptive}
\text{Sun, Q.}, \text{Zhou, W.-X.} and \text{Fan, J.} (2020).
\newblock Adaptive huber regression.
\newblock \textit{Journal of the American Statistical Association},
  \textbf{115} 254--265.

\bibitem[{Wainwright(2019)}]{wainwright2019high}
\text{Wainwright, M.~J.} (2019).
\newblock \textit{High-dimensional statistics: A non-asymptotic viewpoint},
  vol.~48.
\newblock Cambridge University Press.

\bibitem[{Wang and Fan(2022)}]{wang2022robust}
\text{Wang, B.} and \text{Fan, J.} (2022).
\newblock Robust matrix completion with heavy-tailed noise.
\newblock \textit{arXiv preprint arXiv:2206.04276}.

\bibitem[{Wang et~al.(2020)Wang, Peng, Bradic, Li and Wu}]{wang2020tuning}
\text{Wang, L.}, \text{Peng, B.}, \text{Bradic, J.}, \text{Li, R.} and
  \text{Wu, Y.} (2020).
\newblock A tuning-free robust and efficient approach to high-dimensional
  regression.
\newblock \textit{Journal of the American Statistical Association},
  \textbf{115} 1700--1714.

\bibitem[{Wang et~al.(2015)Wang, Peng and Li}]{wang2015high}
\text{Wang, L.}, \text{Peng, B.} and \text{Li, R.} (2015).
\newblock A high-dimensional nonparametric multivariate test for mean vector.
\newblock \textit{Journal of the American Statistical Association},
  \textbf{110} 1658--1669.

\bibitem[{Wang et~al.(2009)Wang, Gerstein and Snyder}]{wang2009rna}
\text{Wang, Z.}, \text{Gerstein, M.} and \text{Snyder, M.} (2009).
\newblock Rna-seq: a revolutionary tool for transcriptomics.
\newblock \textit{Nature reviews genetics}, \textbf{10} 57--63.

\bibitem[{Xia et~al.(2018)Xia, Cai and Li}]{Xia2018}
\text{Xia, Y.}, \text{Cai, T.~T.} and \text{Li, H.} (2018).
\newblock Joint testing and false discovery rate control in high-dimensional
  multivariate regression.
\newblock \textit{Biometrika}, \textbf{105} 249--269.

\bibitem[{Yang et~al.(2017)Yang, Balasubramanian and Liu}]{yang2017high}
\text{Yang, Z.}, \text{Balasubramanian, K.} and \text{Liu, H.} (2017).
\newblock High-dimensional non-gaussian single index models via thresholded
  score function estimation.
\newblock In \textit{International conference on machine learning}. PMLR.

\bibitem[{Yohai and Maronna(1979)}]{yohai1979asymptotic}
\text{Yohai, V.~J.} and \text{Maronna, R.~A.} (1979).
\newblock Asymptotic behavior of m-estimators for the linear model.
\newblock \textit{The Annals of Statistics} 258--268.

\bibitem[{Zhang et~al.(2020)Zhang, Guo, Zhou and Cheng}]{Zhang2020}
\text{Zhang, J.-T.}, \text{Guo, J.}, \text{Zhou, B.} and \text{Cheng, M.-Y.}
  (2020).
\newblock A simple two-sample test in high dimensions based on {$L^2$}-norm.
\newblock \textit{J. Amer. Statist. Assoc.}, \textbf{115} 1011--1027.

\bibitem[{Zheng et~al.(2015)Zheng, Peng and He}]{zheng2015globally}
\text{Zheng, Q.}, \text{Peng, L.} and \text{He, X.} (2015).
\newblock Globally adaptive quantile regression with ultra-high dimensional
  data.
\newblock \textit{Annals of statistics}, \textbf{43} 2225.

\bibitem[{Zhou et~al.(2018)Zhou, Bose, Fan and Liu}]{Zhou2018}
\text{Zhou, W.-X.}, \text{Bose, K.}, \text{Fan, J.} and \text{Liu, H.} (2018).
\newblock A new perspective on robust {$M$}-estimation: finite sample theory
  and applications to dependence-adjusted multiple testing.
\newblock \textit{Ann. Statist.}, \textbf{46} 1904--1931.

\end{thebibliography}

\newpage
\appendix

\section{Appendix}

In the following sections, we present additional simulation results as well as the proofs of all the theoretical results in the main paper. 
\subsection{Additional Simulation}\label{add_simu}
In this section, we conduct additional simulations for FDP control using HL estimator and student's t-statistics, respectively. 
We keep all settings as in \S\ref{num:FDP}, except that the noises are generated from the following case 7 and case 8, where we change the distribution in case 1 and case 6 from $t_3$ to $t_1.$ 
 The numerical outcomes are summarized in Table \ref{tab4} and Table \ref{tab5}.
\begin{itemize}
    \item Case 7: Scaled $t_1$ distribution with $\xi_{i,k}\sim 0.3\cdot t_1,(i,k)\in [n]\times[p].$
    \item Case 8: Scaled $t_1$ distribution, where $\xi_{i,k}\sim 0.3\cdot t_1,(i,k)\in[n]\times [p]$ and $\varepsilon_{j,k}\sim 0.3\cdot t_1, (j,k)\in[n]\times [p].$  
\end{itemize}
\begin{table}[H]
	\begin{center}
		\def\arraystretch{1}
		\setlength\tabcolsep{7pt}
	
		\begin{tabular}{c|c|c c c c c}
			\toprule
		Estimator	& $\alpha$ & $0.05$ & $0.10$  &$0.15$ &$0.20$ &$0.25$ \\
			\hline
		\multirow{2}{*}{HL}	  
			&Case 7  &0.036 (1.000)  &0.082 (1.000)  &0.132 (1.000) & 0.180 (1.000) &0.222 (1.000) \\
&Case 8 &0.069 (1.000)  &0.105 (1.000) &0.168 (1.000) &0.225 (1.000) &0.272 (1.000) \\
\hline
\hline
		\multirow{2}{*}{Student's $t$}	 
			&Case 7  &0.000 (0.282) &0.000 (0.284)  &0.000 (0.320) &0.000 (0.320) &0.000 (0.320)\\
			  &Case 8 &0.000 (0.274)  &0.000 (0.274) &0.000 (0.274) &0.000 (0.322) &0.000 (0.322) \\
			\bottomrule
		\end{tabular}
	\end{center}
	\caption{Empirical FDP and TPP versus the nominal level $\alpha$ of the test via HL estimator and student's $t$-statistics when $\mu=0.5$. The remain captions are the same with those in Table \ref{tab2}. }
	\label{tab4}
\end{table}

\begin{table}[H]
	\begin{center}
		\def\arraystretch{1}
		\setlength\tabcolsep{7pt}
	
		\begin{tabular}{c|c|c c c c c}
			\toprule
		Estimator	& $\alpha$ & $0.05$ & $0.10$  &$0.15$ &$0.20$ &$0.25$ \\
			\hline
		\multirow{2}{*}{HL}	  
			&Case 7  &0.036 (1.000)  &0.82 (1.000)  &0.122 (1.000) & 0.179 (1.000) &0.224 (1.000) \\
&Case 8 &0.058 (1.000)  &0.105 (1.000) &0.157 (1.000) &0.214 (1.000) &0.252 (1.000) \\
\hline
\hline
		\multirow{2}{*}{Student's $t$}	 
			&Case 7  &0.000 (0.282) &0.000 (0.282)  &0.000 (0.282) &0.000 (0.284) &0.000 (0.284)\\
			  &Case 8  &0.000 (0.134) &0.000 (0.134)  &0.000 (0.134) &0.000 (0.134) &0.000 (0.134)\\
			\bottomrule
		\end{tabular}
	\end{center}
	\caption{Empirical FDP and TPP versus the nominal level $\alpha$ of the test via HL estimator and student's $t$-statistics when $\mu=0.3$. The other captions are the same with those in Table \ref{tab2}. }
	\label{tab5}
\end{table}
We conclude from Table \ref{tab4} and Table \ref{tab5}, the HL estimator outperforms the student's $t$-statistics when the noise follows $t_1$ distribution in terms of the FDP and TPP. Therefore, this further confirms the robustness of HL estimator.

\section{Lemmas}
\begin{lemma}
\label{Lemma_Hoeffding_inequality}
Let $S_{n} = \sum_{i = 1}^{n} Y_{i}$, where $Y_{1}, \ldots, Y_{n} \in \mathbb{R}$ are independent random variables such that $a_{i} \leq Y_{i} \leq b_{i}$ for each $i \in [n]$. Then, for any $z > 0$, we have 
\begin{align*}
    \PP(|S_{n} - \EE S_{n}| > z) \leq 2 \exp\left\{-\frac{2z^{2}}{\sum_{i = 1}^{n} (b_{i} - a_{i})^{2}}\right\}.
\end{align*}
\end{lemma}

\begin{lemma}
\label{Lemma_Hoeffding_inequality_U_statistics}
Let $Y_{1}, \ldots, Y_{n} \in \mathbb{R}$ be i.i.d.~random variables and $H_{n} = \{n(n - 1)\}^{-1} \sum_{i \neq j \in [n]} h(X_{i}, X_{j})$, where $h(x, y)$ is a symmetric function with $a \leq h(x, y) \leq b$ for some $a \leq b \in \mathbb{R}$. Then, for any $z > 0$, we have 
\begin{align*}
    \mathbb{P}(H_{n} - \mathbb{E} H_{n} > z) \leq \exp\left\{-\frac{nz^{2}}{(b - a)^{2}}\right\}.
\end{align*}
Let $V_{1}, \ldots, V_{m} \in \mathbb{R}$ be another sample of i.i.d.~random variables independent of $\{Y_{1}, \ldots, Y_{n}\}$ and let $\mathcal{H}_{n, m} = (nm)^{-1} \sum_{i = 1}^{n} \sum_{j = 1}^{m} \bar{h}(Y_{i}, V_{j})$, where $\bar{h}(x, y)$ is bounded such that $a \leq \bar{h}(x, y) \leq b$ for some $a \leq b \in \mathbb{R}$. Then, for any $z > 0$, we have 
\begin{align*}
    \mathbb{P}(\mathcal{H}_{n, m} - \mathbb{E} \mathcal{H}_{n, m} > z) \leq \exp\left\{-\frac{2(n \wedge m) z^{2}}{(b - a)^{2}}\right\}.
\end{align*}
\end{lemma}

\section{Proof of Theoretical Results in \S\ref{sec:nonasym}}

\subsection{Proof of Theorem~\ref{Theorem_quantile_consistency_one}}
For simplicity of notation, we write $W_{n}(t) = U_{n}(\theta + t) - U(\theta + t)$ for $t \in \RR$. Define $R(X_{i}, X_{j}, t) = \mathbb{I}\{\theta < (X_{i} + X_{j})/2 \leq \theta + t\}$ and 
\begin{align*}
    \bar{R}(X_{i}, X_{j}, t) = R(X_{i}, X_{j}, t) - \EE\{R(X_{i}, X_{j}, t)|X_{i}\} - \EE\{R(X_{i}, X_{j}, t)|X_{j}\} + \EE\{R(X_{i}, X_{j}, t)\}. 
\end{align*}
With this notation, we have $W_{n}(t) - W_{n}(0) = \bar{R}_{n}^{\diamond}(t) + \bar{R}_{n}^{\circ}(t)$, where
\begin{align*}
    \bar{R}_{n}^{\diamond}(t) = \frac{2}{n}\sum_{i = 1}^{n} [\EE\{R(X_{i}, X_{j}, t)|X_{i}\} - \EE\{R(X_{i}, X_{j}, t)\}] \enspace \mathrm{and} \enspace \bar{R}_{n}^{\circ}(t) = \frac{1}{n(n - 1)} \sum_{i \neq j \in [n]} \bar{R}(X_{i}, X_{j}, t).
\end{align*}

\begin{lemma}
\label{Lemma_bound_Rn_circ}
For any $0 < t \leq 1/(2\|f\|_{\infty})$, there exists a universal positive constant $C$ such that 
\begin{align}
\label{eq_Rn_circ_concentration}
    \PP(|\bar{R}_{n}^{\circ}(t)| > z) \leq C \exp\bigg[-\frac{1}{C}\min\bigg\{\frac{n^{2}z^{2}}{t \|f\|_{\infty}}, \bigg(\frac{n^{3}z^{2}}{t \|f\|_{\infty}}\bigg)^{1/3}, \frac{n z}{\sqrt{t \|f\|_{\infty}}}, nz^{1/2}\bigg\}\bigg]. 
\end{align}
\end{lemma}

\begin{proof}[Proof of Lemma~\ref{Lemma_bound_Rn_circ}]
Observe that $\bar{R}(X_{i}, X_{j}, t)$ is a bounded canonical kernel of $X_{i}$ and $X_{j}$ for any $t \in \RR$. Hence it is straightforward to derive~\eqref{eq_Rn_circ_concentration} by Theorem 3.3 in~\citet{gine2000exponential}. 
\end{proof}

\begin{lemma}
\label{Lemma_Oscillation_W_process}
For any $0 < \Lambda \leq 1/(2\|f\|_{\infty})$, with probability at least $1 - C \exp(-z)$, we have 
\begin{align}
\label{eq_W_W_concentration}
    \sup_{|t| \leq \Lambda} |W_{n}(t) - W_{n}(0)| \leq C_{1} \bigg\{\bigg(\frac{K + z}{n}\bigg)^{2} + \Lambda \|f\|_{\infty}\bigg(\frac{z + C_{2}}{n}\bigg)^{1/2} + \frac{K + z}{n} (\Lambda \|f\|_{\infty})^{1/2}\bigg\},
\end{align}
where $K = \min\{k \in \NN : 2^{k} \geq n\}$. 
\end{lemma}

\begin{proof}[Proof of Lemma~\ref{Lemma_Oscillation_W_process}]
For any $h \in [2^{K}]$ and $t \in (\Lambda (h - 1) 2^{-K}, \Lambda h 2^{-K}]$, we have 
\begin{align*}
    W_{n}(t) - W_{n}(0) &\leq U_{n}\left(\theta + \frac{\Lambda h}{2^{K}}\right) - U\left(\theta + \frac{\Lambda (h - 1)}{2^{K}}\right) - W_{n}(0)\cr
    &\leq W_{n}\left(\frac{\Lambda h}{2^{K}}\right) - W_{n}(0) + U\left(\theta + \frac{\Lambda h}{2^{K}}\right) - U\bigg(\theta + \frac{\Lambda (h - 1)}{2^{K}}\bigg)\cr
    &\leq W_{n}\left(\frac{\Lambda h}{2^{K}}\right) - W_{n}(0) + \frac{\Lambda \|U'\|_{\infty}}{2^{K}}.
\end{align*}
Similarly, we have 
\begin{align*}
    W_{n}(t) - W_{n}(0) \geq W_{n}\bigg(\frac{\Lambda (h - 1)}{2^{K}}\bigg) - W_{n}(0) - \frac{\Lambda \|U'\|_{\infty}}{2^{K}}. 
\end{align*}
Consequently, we obtain 
\begin{align*}
    \sup_{0 < t \leq \Lambda} |W_{n}(t) - W_{n}(0)| &\leq \max_{h \in [2^{K}]} \left|W_{n}\left(\frac{\Lambda h}{2^{K}}\right) - W_{n}(0)\right| + \frac{\Lambda \|U'\|_{\infty}}{2^{K}}\cr
    &\leq \underbrace{\max_{h \in [2^{K}]}\left|\bar{R}_{n}^{\circ}\left(\frac{\Lambda h}{2^{K}}\right)\right|}_{\Gamma_{K}^{\circ}} + \underbrace{\max_{h \in [2^{K}]}\left|\bar{R}_{n}^{\diamond}\left(\frac{\Lambda h}{2^{K}}\right)\right|}_{\Gamma_{K}^{\diamond}} + \frac{\Lambda \|U'\|_{\infty}}{2^{K}}.  
\end{align*}
For each $h \in [2^{K}]$, by~\eqref{eq_Rn_circ_concentration}, with probability at least $1 - C\exp(-z)$, 
\begin{align*}
    \Gamma_{K}^{\circ} \leq C \bigg\{\bigg(\frac{K + z}{n}\bigg)^{2} + \frac{K + z}{n} \sqrt{\Lambda \|U'\|_{\infty}}\bigg\}, 
\end{align*}
where $C > 0$ is a universal constant. 
\begin{align*}
    \Gamma_{K}^{\diamond} \leq \sum_{k = 1}^{K} \max_{h \in [2^{k}]}\bigg|\bar{R}_{n}^{\diamond}\bigg(\frac{\Lambda h}{2^{k}}\bigg) - \bar{R}_{n}^{\diamond}\bigg(\frac{\Lambda (h - 1)}{2^{k}}\bigg)\bigg| + \max_{h \in [2]} \bigg|\bar{R}_{n}^{\diamond}\bigg(\frac{\Lambda h}{2}\bigg)\bigg| =: \Gamma_{K, 1}^{\diamond} + \Gamma_{K, 2}^{\diamond}. 
\end{align*}
For each $k \in [K]$, by Lemma~\ref{Lemma_Hoeffding_inequality}, with probability at least $1 - \exp(-z)$, we have
\begin{align*}
    \max_{h \in [2^{k}]} \bigg|\bar{R}_{n}^{\diamond}\left(\frac{\Lambda h}{2^{k}}\right) - \bar{R}_{n}^{\diamond}\bigg(\frac{\Lambda (h - 1)}{2^{k}}\bigg)\bigg| \leq C \Lambda \|f\|_{\infty} 2^{-k}\sqrt{\frac{k + z}{n}}. 
\end{align*}
Hence with the same probability, we have
\begin{align*}
    \Gamma_{K, 1}^{\diamond} \leq \sum_{k = 1}^{K} \bigg(C \Lambda \|f\|_{\infty} 2^{-k}\sqrt{\frac{2k + z}{n}}\bigg) \leq C_{1} \Lambda \|f\|_{\infty} \sqrt{\frac{z + C_{2}}{n}}. 
\end{align*}
Similarly, it follows that
\begin{align*}
    \PP\left(\Gamma_{K, 2}^{\diamond} > C_{3}\Lambda \|f\|_{\infty} \sqrt{(2z)/n}\right) \leq C_{4} \exp(-z). 
\end{align*}
Putting all these pieces together, we obtain~\eqref{eq_W_W_concentration}. 
\end{proof}

\begin{proof}[Proof of Theorem~\ref{Theorem_quantile_consistency_one}]
For any $z \geq 0$, since $U(t)$ is a cumulative distribution function, we have 
\begin{align*}
    \frac{1}{2} > U_{n}(\theta &+ z) = U(\theta + z) - \{U(\theta + z) - U_{n}(\theta + z)\}\cr
    &\geq U(\theta) + \int_{0}^{z \wedge c_{0}} U'(\theta + \nu) d\nu - \{U(\theta + z) - U_{n}(\theta + z)\}\cr
    &\geq \frac{1}{2} + \kappa_{0}(z \wedge c_{0}) - \{U(\theta + z) - U_{n}(\theta + z)\}.  
\end{align*}
By the definition of $\hat{\theta}$ in~\eqref{eq_def_mu_hat_quantile}, it follows that $\PP(\hat{\theta} > \theta + z) \leq \PP\{1/2 > U_{n}(\theta + z)\}$. Therefore 
\begin{align*}
    \PP(\hat{\theta} > \theta + z) \leq \PP\{U(\theta + z) - U_{n}(\theta + z) > \kappa_{0} (z \wedge c_{0})\} \leq \exp\{-n\kappa_{0}^{2}(z\wedge c_{0})^{2}\},
\end{align*}
where the last inequality follows from the Hoeffding inequality in Lemma~\ref{Lemma_Hoeffding_inequality_U_statistics}.

Recall that $W_{n}(t) = U_{n}(\theta + t) - U(\theta + t)$ for $t \in \RR$. Taking $\Lambda = \sqrt{z/(n \kappa_{0}^{2})}$ yields $\PP(|\hat{\theta} - \theta| > \Lambda) \leq 2 \exp(-z)$. By Lemma~\ref{Lemma_Oscillation_W_process}, with probability at least $1 - C \exp(-z)$, we have  
\begin{align*}
    \sup_{|\delta| \leq \Lambda} |W_{n}(\delta) - W_{n}(0)| \leq \frac{C_{1}\|f\|_{\infty}(z + c)}{n\kappa_{0}}. 
\end{align*}
We then obtain
\begin{align*}
    \bigg|\hat{\theta} - \theta - \frac{U_{n}(\hat{\theta}) - U_{n}(\theta)}{U'(\theta)}\bigg| \leq \frac{|W_{n}(\hat{\theta} - \theta) - W_{n}(0)|}{\kappa_{0}} + \frac{\kappa_{1} |\hat{\theta} - \theta|^{2}}{\kappa_{0}} \leq \frac{C_{1} \|f\|_{\infty}(z + c)}{n\kappa_{0}^{2}} + \frac{\kappa_{1} z}{n \kappa_{0}^{3}}. 
\end{align*}
Finally, denote 
\begin{align*}
    \Delta_{\theta} :&= \frac{1/2 - U_{n}(\theta)}{U'(\theta)} - \frac{2}{nU'(\theta)} \sum_{i = 1}^{n} \left\{\frac{1}{2} - F(-\xi_{i})\right\}\cr
    &= \frac{1}{n(n - 1)U'(\theta)} \sum_{i \neq j \in [n]} \left(F(-\xi_{i}) + F(-\xi_{j}) - \mathbb{I}\{\xi_{i} + \xi_{j} \leq 0\} - \frac{1}{2}\right).
\end{align*}
Observe that $|F(-\xi_{i}) + F(-\xi_{j}) - \mathbb{I}\{\xi_{i} + \xi_{j} \leq 0\} - 1/2| \leq 2$ uniformly for $i \neq j \in [n]$. Then~\eqref{eq_Hoeffding_decomposition_Bahadur_representation} follows from Theorem 3.3 in~\citet{gine2000exponential}. 
\end{proof}

\subsection{Proof of Theorem~\ref{Theorem_Berry-Esseen_one}}
\begin{proof}[Proof of Theorem~\ref{Theorem_Berry-Esseen_one}]
For simplicity of notation, denote 
\begin{align}
\label{eq_definition_Tn_Tn_sharp}
    T_{n} = \frac{\sqrt{n}(\hat{\theta} - \theta)}{\sigma_{\theta}}\enspace \mathrm{and} \enspace T_{n}^{\sharp} = \frac{1}{\sqrt{n\Var\{F(-\xi_{1})\}}} \sum_{i = 1}^{n} \left\{\frac{1}{2} - F(-\xi_{i})\right\}.
\end{align}
Note that $\max_{i \in [n]}|1/2 - F(-\xi_{i})| \leq 1/2$. Hence $\sup_{z \in \mathbb{R}} |\mathbb{P}(T_{n}^{\sharp} \leq z) - \Phi(z)| \leq C n^{-1/2}$. Let $\mathcal{C} < \infty$ be a sufficiently large positive constant. Then it follows from~\eqref{eq_Hoeffding_decomposition_Bahadur_representation} that
\begin{align*}
    \mathbb{P}\left(|T_{n} - T_{n}^{\sharp}| > \frac{\mathcal{C}\log n}{\sqrt{n}}\right) \leq \frac{C_{1}}{\sqrt{n}}.
\end{align*}
Consequently, we obtain 
\begin{align*}
    \sup_{z \in \mathbb{R}} |\PP(T_{n} \leq z) - \Phi(z)| \leq \sup_{z \in \RR} |\PP(T_{n}^{\sharp} \leq z) - \Phi(z)| + \sup_{z \in \RR} |\PP(T_{n} \leq z) - \PP(T_{n}^{\sharp} \leq z)| \leq \mathcal{C} (\log n)/\sqrt{n}.  
\end{align*}
\end{proof}

\subsection{Proof of Theorem~\ref{Theorem_Cramer_one}}
\begin{proof}[Proof of Theorem~\ref{Theorem_Cramer_one}]
Recall the definitions of $T_{n}$ and $T_{n}^{\sharp}$ in~\eqref{eq_definition_Tn_Tn_sharp}. Since $\max_{i \in [n]} |1/2 - F(-\xi_{i})| \leq 1/2$, it follows that
\begin{align*}
    \left|\frac{\mathbb{P}(T_{n}^{\sharp} > z)}{1 - \Phi(z)} - 1\right| \leq \frac{C(1 + z^{3})}{\sqrt{n}} \enspace \mathrm{for} \enspace 0 \leq z \leq C_{0} n^{1/6}, 
\end{align*}
where $C_{0} > 0$ is any fixed constant and $C$ is a positive constant depending only on $C_{0}$. Denote $\bar{\Phi}(\cdot) = 1 - \Phi(\cdot)$. By~\eqref{eq_Hoeffding_decomposition_Bahadur_representation}, we have 
\begin{align*}
    \mathbb{P}(T_{n} > z) &\leq \mathbb{P}(T_{n}^{\sharp} > z - \delta_{n}) + C_{1} \exp(-\sqrt{n} \delta_{n})\cr
    &\leq \left\{1 + \frac{C(1 + z^{3})}{\sqrt{n}}\right\}\bar{\Phi}(z - \delta_{n}) + C_{1} \exp(-\sqrt{n} \delta_{n})\cr
    &\leq \left\{1 + \frac{C(1 + z^{3})}{\sqrt{n}}\right\}\{1 + (1 + z)\delta_{n}\exp(z \delta_{n})\}\bar{\Phi}(z) + C_{1} \exp(-\sqrt{n} \delta_{n}). 
\end{align*}
By Mill's inequality, for any $z > 0$, 
\begin{align*}
    z\sqrt{2\pi} \exp\left(\frac{z^{2}}{2}\right) \leq \frac{1}{\bar{\Phi}(z)} \leq 2(z \vee 1) \sqrt{2\pi} \exp\left(\frac{z^{2}}{2}\right). 
\end{align*}
Consequently, it follows that 
\begin{align*}
    \frac{\mathbb{P}(T_{n} > z)}{1 - \Phi(z)} - 1 \leq \frac{C(1 + z^{3})}{\sqrt{n}} + C_{2}(1 + z) \delta_{n} + 2 C_{1}(z\vee 1) \sqrt{2\pi} \exp\left(\frac{z^{2}}{2} - \sqrt{n} \delta_{n}\right). 
\end{align*}
Similarly, we have 
\begin{align*}
    1 - \frac{\mathbb{P}(T_{n} > z)}{1 - \Phi(z)} \leq \frac{C(1 + z^{3})}{\sqrt{n}} + C_{2}(1 + z) \delta_{n} + C_{1} z \sqrt{2\pi} \exp\left(\frac{z^{2}}{2} - \sqrt{n}\delta_{n}\right). 
\end{align*}
Putting all these pieces together, we obtain~\eqref{eq_moderate_deviation_mu_one_sample}. 
\end{proof}

\subsection{Proof of Theorem~\ref{Theorem_quantile_consistency_two}}
Define the two-sample~\emph{U}-process $\mathcal{U}_{n, m}(t) = (nm)^{-1} \sum_{i = 1}^{n} \sum_{j = 1}^{m} \mathbb{I}\{X_{i} - Y_{j} \leq t\}$. Then the HL estimator $\hat{\Theta}$ in~\eqref{eq_HL_estimator_two} can be equivalently expressed as the sample median of the \emph{U}-process $\mathcal{U}_{n, m}(t)$, namely,
\begin{align}
\label{eq_def_mu_hat_quantile}
    \hat{\Theta} = \inf\{t \in \RR : \mathcal{U}_{n, m}(t) \geq 1/2\}. 
\end{align}
The remaining steps for proving Theorem~\ref{Theorem_quantile_consistency_two} can be derived by following similar steps as in the proof of Lemma~\ref{Lemma_Oscillation_W_process} and proof of Theorem~\ref{Theorem_quantile_consistency_one}.

\subsection{Proof of Theorem~\ref{Theorem_Berry_Esseen_two}}
For simplicity of notation, we write $\mathcal{T}_{N} = \sqrt{N}(\hat{\Theta} - \Theta)/\tilde{\Theta}$ and 
\begin{align*}
    \mathcal{T}_{N}^{\sharp} = \frac{\bar{G}_{n} - \bar{F}_{m}}{\sqrt{\mathrm{Var}(\bar{G}_{n} - \bar{F}_{m})}}, 
\end{align*}
where $\bar{G}_{n} = n^{-1} \sum_{i = 1}^{n} G(\xi_{i})$ and $\bar{F}_{m} = m^{-1} \sum_{j = 1}^{m} F(\varepsilon_{j})$. Then, similar to proof of Theorem~\ref{Theorem_Berry-Esseen_one}, it follows that  
\begin{align*}
    \sup_{z \in \mathbb{R}} |\mathbb{P}(\mathcal{T}_{N} \leq z) - \Phi(z)| \leq \sup_{z \in \mathbb{R}} |\mathbb{P}(\mathcal{T}_{N} \leq z) - \mathbb{P}(\mathcal{T}_{N}^{\sharp} \leq z)| + \sup_{z \in \mathbb{R}} |\mathbb{P}(\mathcal{T}_{N}^{\sharp} \leq z) - \Phi(z)| \lesssim \frac{\log N}{\sqrt{N}}. 
\end{align*} 
The proof of the Cram\'{e}r-type moderate deviation for $\hat{\Theta}$ can be derived by following similar proof procedures of Theorem~\ref{Theorem_Cramer_one}. Therefore, we decide to omit the details. 

\section{Proof of Results in \S\ref{bootstrap}}
Define  
\begin{align*}
    \hat{\sigma}_{\theta}^{2} = \frac{4}{n \{U'(\theta)\}^{2}} \sum_{i = 1}^{n} \bigg(\frac{1}{n - 1} \sum_{j \neq i} \mathbb{I}\{\xi_{i} + \xi_{j} \leq 0\} - U_{n}(\theta)\bigg)^{2}. 
\end{align*}

\begin{lemma}
\label{Lemma_sample_sigma_consistency}
Under the conditions of Theorem~\ref{Theorem_quantile_consistency_one}, we have 
\begin{align}\label{eq_concentration_inequality_sigma_estimation}
    \mathbb{P}\bigg(\bigg|\frac{\hat{\sigma}_{\theta}}{\sigma_{\theta}} - 1\bigg| > z\bigg) \lesssim \exp(-C_{U} n) + \exp(-C_{U} n z) + \exp(-C_{U} n z^{2}),
\end{align}
where $C_{U}$ is a positive constant depending only on $U$. 
\end{lemma}

\begin{proof}[Proof of Lemma~\ref{Lemma_sample_sigma_consistency}]
For each $i \in [n]$, denote  
\begin{align*}
    W_{i} = \frac{1}{n - 1}\sum_{j \neq i} \mathbb{I}\{\xi_{i} + \xi_{j} \leq 0\} - F(-\xi_{i}).
\end{align*}
Recall that $U_{n}(\theta) = \{n(n - 1)\}^{-1} \sum_{i \neq j \in [n]} \mathbb{I}\{\xi_{i} + \xi_{j} \leq 0\}$. Hence 
\begin{align*}
    \hat{\sigma}_{\theta}^{2} -& \underbrace{\frac{4}{n\{U'(\theta)\}^{2}} \sum_{i = 1}^{n} \left\{F(-\xi_{i}) - \frac{1}{2}\right\}^{2}}_{\tilde{\sigma}_{\theta}^{2}} = \frac{4}{n\{U'(\theta)\}^{2}} \sum_{i = 1}^{n} W_{i}^{2} - \frac{4|U_{n}(\theta) - 1/2|^{2}}{\{U'(\theta)\}^{2}}\cr
    &+ \frac{8}{n\{U'(\theta)\}^{2}} \sum_{i = 1}^{n} W_{i} \left\{F(-\xi_{i}) - \frac{1}{2}\right\} =: \Delta_{\sigma, 1} - \Delta_{\sigma, 2} + \Delta_{\sigma, 3}. 
\end{align*}
Since $\mathbb{E}(\tilde{\sigma}_{\theta}^{2}) = \sigma_{\theta}^{2}$ and $\max_{i \in [n]}|F(-\xi_{i}) - 1/2| \leq 1/2$, it follows from Lemma~\ref{Lemma_Hoeffding_inequality} that 
\begin{align*}
    \mathbb{P} (|\tilde{\sigma}_{\theta}^{2} - \sigma_{\theta}^{2}| > z) \leq 2\exp(-C_{U} n z^{2}),  
\end{align*}
where $C_{U}$ is a positive constant depending only on the density function $U'$.
Note that $\{\mathbb{I}\{\xi_{i} + \xi_{j} \leq 0\} - F(-\xi_{i})\}_{j \neq i}$ is a sequence of independent centered random variables conditional on $\xi_{i}$ for each $i \in [n]$. Hence, by Hoeffding's inequality, we have $\mathbb{P}(\max_{i \in [n]} W_{i}^{2} > z) \leq 2 n \exp(- C n z)$. Therefore 
\begin{align*}
    \mathbb{P}(\Delta_{\sigma, 1} > z) \leq \mathbb{P}\bigg(\frac{4}{\{U'(\theta)\}^{2}} \max_{i \in [n]} W_{i}^{2} > z\bigg) \leq 2n\exp(-C_{U}nz). 
\end{align*}
Consequently, by the Cauchy-Schwarz inequality, we have $|\Delta_{\sigma, 3}|^{2} \leq 4 \tilde{\sigma}_{\theta}^{2} \Delta_{\sigma, 1}$ and 
\begin{align*}
    \mathbb{P}(|\Delta_{\sigma, 4}| > z) \leq 2\exp(-C_{U} n) + 2n \exp(-C_{U} n z^{2}).
\end{align*}
By Lemma~\ref{Lemma_Hoeffding_inequality_U_statistics}, we have $\mathbb{P}(|U_{n}(\theta) - 1/2| > z) \leq 2\exp(-nz^{2})$ and $\mathbb{P}(\Delta_{\sigma, 3} > z) \leq 2 \exp(-C_{U} n z)$. Putting all these pieces together, we obtain~\eqref{eq_concentration_inequality_sigma_estimation}. 
\end{proof}

\subsection{Proof of Theorem~\ref{Theorem_quantile_consistency_one_Bootstrap}}
Define the bootstrap~\emph{U}-process 
\begin{align*}
    U_{n}^{\star}(t) = \frac{1}{B_{n}} \sum_{i \neq j \in [n]} \omega_{i} \omega_{j} \mathbb{I}\left\{\frac{X_{i} + X_{j}}{2} \leq t\right\}, \enspace \mathrm{where} \enspace B_{n} = \sum_{i \neq j \in [n]} \omega_{i} \omega_{j}. 
\end{align*}
Define $W_{n}^{\star}(t) = U_{n}^{\star}(\theta + t) - U(\theta + t)$ and $W_{n}^{\diamond}(t) = U_{n}^{\star}(\theta + t) - U_{n}(\theta + t)$ for $t \in \mathbb{R}$. We first introduce some notation. Let $K = \min\left\{k \in \mathbb{N} : 2^{k} \geq n\right\}$. For each $i \neq j \in [n]$ and $h \in [2^{K}]$, denote 
\begin{align*}
    \mathcal{R}\left(X_{i}, X_{j}, \frac{\Lambda h}{2^{K}}\right) &= R\left(X_{i}, X_{j}, \frac{\Lambda h}{2^{K}}\right) - U_{n}\left(\theta + \frac{\Lambda h}{2^{K}}\right) + U_{n}(\theta)\cr
    &= R\left(X_{i}, X_{j}, \frac{\Lambda h}{2^{K}}\right) - \frac{1}{n(n - 1)} \sum_{i \neq j \in [n]} R\left(X_{i}, X_{j}, \frac{\Lambda h}{2^{K}}\right). 
\end{align*}

\begin{lemma}
\label{Lemma_W_n_star_process}
Let $0 <\Lambda \leq 1/(2\|f\|_{\infty})$. Under $\mathcal{E}_{K} = \{\Gamma_{n} \leq \Gamma_{\diamond}\}$, for any $z \leq Cn$, we have 
\begin{align}
\label{eq_Wn_star_bound}
    \mathbb{P}^{\star}&\bigg\{\sup_{|t| \leq \Lambda} |W_{n}^{\star}(t) - W_{n}^{\star}(0)| > \sup_{|t| \leq \Lambda} |W_{n}(t) - W_{n}(0)| + \frac{\Lambda \|U'\|_{\infty}}{2^{K}} + \frac{C_{1}n(z + K) + C_{2} \sqrt{\Gamma_{\diamond} (z + K)}}{2n^{2}}\bigg\}\cr
    &\leq C \exp(-z) + \exp\left(-\frac{n}{16}\right).
\end{align}
\end{lemma}

\begin{proof}[Proof of Lemma~\ref{Lemma_W_n_star_process}]
For any $h \in [2^{K}]$ and $t \in (\Lambda(h - 1)2^{-K}, \Lambda h 2^{-K}]$, we have 
\begin{align*}
    W_{n}^{\star}(t) - W_{n}^{\star}(0) &= U_{n}^{\star}(\theta + t) - U(\theta + t) - \{U_{n}^{\star}(\theta) - U(\theta)\}\cr
    &\leq U_{n}^{\star}\left(\theta + \frac{\Lambda h}{2^{K}}\right) - U\left(\theta + \frac{\Lambda(h - 1)}{2^{K}}\right) - \{U_{n}^{\star}(\theta) - U(\theta)\}\cr
    &\leq W_{n}^{\star}\left(\frac{\Lambda h}{2^{K}}\right) - W_{n}^{\star}(0) + \frac{\Lambda \|U'\|_{\infty}}{2^{K}}\cr
    &\leq W_{n}^{\diamond}\left(\frac{\Lambda h}{2^{K}}\right) - W_{n}^{\diamond}(0) + \sup_{|t| \leq \Lambda} |W_{n}(t) - W_{n}(0)| + \frac{\Lambda \|U'\|_{\infty}}{2^{K}} 
\end{align*}
Similarly, we have 
\begin{align*}
    W_{n}^{\star}(t) - W_{n}^{\star}(0) \geq W_{n}^{\diamond}\left(\frac{\Lambda(h - 1)}{2^{K}}\right) - W_{n}^{\diamond}(0) - \sup_{|t| \leq \Lambda}|W_{n}(t) - W_{n}(0)| - \frac{\Lambda \|U'\|_{\infty}}{2^{K}}.
\end{align*}
Consequently, we have 
\begin{align*}
    \sup_{0 < t \leq \Lambda} |W_{n}^{\star}(t) - W_{n}(0)| \leq \max_{h \in [2^{K}]} \bigg|W_{n}^{\diamond}\bigg(\frac{\Lambda h}{2^{K}}\bigg) - W_{n}^{\diamond}(0)\bigg| + \sup_{|t| \leq \Lambda} |W_{n}(t) - W_{n}(0)| + \frac{\Lambda \|U'\|_{\infty}}{2^{K}}
\end{align*}
Hence it suffices to upper bound $\max_{h \in [2^{K}]} |W_{n}^{\diamond}(\Lambda h/2^{K}) - W_{n}^{\diamond}(0)|$. For each $h \in [2^{K}]$, we have 
\begin{align*}
    W_{n}^{\diamond}\bigg(\frac{\Lambda h}{2^{K}}\bigg) - W_{n}^{\diamond}(0) &= \frac{1}{N} \sum_{i \neq j \in [n]} \mathcal{R}\left(X_{i}, X_{j}, \frac{\Lambda h}{2^{K}}\right) (\omega_{i} \omega_{j} - 1)\cr
    &= \frac{1}{N} \sum_{i \neq j \in [n]} \mathcal{R}\left(X_{i}, X_{j}, \frac{\Lambda h}{2^{K}}\right) (\omega_{i} - 1) (\omega_{j} - 1)\cr
    &+ \frac{2}{N} \sum_{j = 1}^{n} \bigg\{\sum_{i \neq j} \mathcal{R}\left(X_{i}, X_{j}, \frac{\Lambda h}{2^{K}}\right)\bigg\} (\omega_{j} - 1)\cr
    &=: \Delta_{W, h}^{\natural} + \Delta_{W, h}^{\sharp}.
\end{align*}
Recall that $B_{n} = \sum_{i \neq j \in [n]} \omega_{i} \omega_{j}$, where $\omega_{1}, \ldots, \omega_{n}$ are i.i.d.~random variables such that $0\leq \omega_{i}\leq 2$ and $\mathbb{E}(\omega_{i}) = 1$ for each $i \in [n]$. Hence, it follows from Lemma~\ref{Lemma_Hoeffding_inequality_U_statistics} that $\mathbb{P}(B_{n} > 2n(n - 1)) \leq \exp(-n/16)$. We first upper bound $\max_{h \in [2^{K}]}|\Delta_{W, h}^{\natural}|$. Observe that $\Delta_{W, h}^{\natural}$ is a degenerate second order~\emph{U}-statistic with  
\begin{align*}
    \max_{i\neq j\in [n]}\max_{h \in [2^{K}]} \bigg|\mathcal{R}\bigg(X_{i}, X_{j}, \frac{\Lambda h}{2^{K}}\bigg)\bigg| \leq 1.
\end{align*}
Hence, by Theorem 3.3 in~\citet{gine2000exponential},
\begin{align*}
    \mathbb{P}^{\star}\bigg(\max_{h \in [2^{K}]} |\Delta_{W, h}^{\natural}| > \frac{z}{2n^{2}}\bigg) \leq C 2^{K} \exp\bigg[-\frac{1}{C}\min\bigg\{\frac{z^{2}}{n^{2}}, \frac{z}{n}, \bigg(\frac{z}{\sqrt{n}}\bigg)^{2/3}, \sqrt{z}\bigg\}\bigg] + \exp\left(-\frac{n}{16}\right).
\end{align*}
We now upper bound $\max_{h \in [2^{K}]} |\Delta_{W, h}^{\sharp}|$. For simplicity of notation, denote $R_{ij, h} = R (X_{i}, X_{j}, \Lambda h/2^{K})$ and $\bar{R}_{n, h} = \{n(n - 1)\}^{-1} \sum_{i \neq j \in [n]} R_{ij, h}$. Define  
\begin{align}
\label{eq_definition_Gamma_n}
    \Gamma_{n} = \max_{h \in [2^{K}]}\sum_{j = 1}^{n} \bigg|\sum_{i \neq j} \mathcal{R}\bigg(X_{i}, X_{j}, \frac{\Lambda h}{2^{K}}\bigg)\bigg|^{2}.
\end{align}
Hence, under $\mathcal{E}_{K}$, by Theorem 1.1 in~\citet{bentkus2015tight}, it follows that 
\begin{align*}
    \mathbb{P}^{\star}\bigg\{\max_{h \in [2^{K}]} |\Delta_{W, h}^{\sharp}| > \frac{C_{1}\sqrt{\Gamma_{\diamond}(z + K)}}{2n^{2}}\bigg\} \leq C_{2}\exp(-z) + \exp\left(-\frac{n}{16}\right).
\end{align*}
Putting all these pieces together, we obtain~\eqref{eq_Wn_star_bound}.
\end{proof}

\begin{lemma}
\label{Lemma_concentration_Gamma_n}
Let $\Gamma_{n}$ be defined in~\eqref{eq_definition_Gamma_n}. Assume that $\Lambda \leq 1/(2\|f\|_{\infty})$ and $\log n \leq n\Lambda \|f\|_{\infty}$. Then we have 
\begin{align}
\label{eq_upper_bound_Gamma}
    \mathbb{P}(\Gamma_{n} \leq C n^{3}\Lambda^{2} \|f\|_{\infty}^{2}) \geq 1 - C_{1}\exp(-n\Lambda \|f\|_{\infty}). 
\end{align}
\end{lemma}

\begin{proof}[Proof of Lemma~\ref{Lemma_concentration_Gamma_n}]
By the triangle inequality, 
\begin{align*}
    \Gamma_{n} &\lesssim \max_{h \in [2^{K}]} \sum_{j = 1}^{n} \bigg|\sum_{i \neq j} \{R_{ij, h} - \mathbb{E} (R_{ij, h}|X_{j})\}\bigg|^{2} + \max_{h \in [2^{K}]} n^{2} \sum_{j = 1}^{n} |\mathbb{E}(R_{ij, h}|X_{j}) - \mathbb{E}(R_{ij, h})|^{2}\cr
    &+ \max_{h \in [2^{K}]} n^{3} |\bar{R}_{n, h} - \mathbb{E}(\bar{R}_{n, h})|^{2} =: \Gamma_{n, 1} + \Gamma_{n, 2} + \Gamma_{n, 3}.
\end{align*}
For each $j \in [n]$, conditional on $X_{j}$, $\{R_{ij, h} - \mathbb{E}(R_{ij, h}|X_{j})\}_{i \neq j}$ are independent centered random variables with  
\begin{align*}
    \max_{h \in [2^{K}]} \max_{j \in [n]} |R_{ij, h} - \mathbb{E}(R_{ij, h}|X_{j})| \leq 1 \enspace \mathrm{and} \enspace \max_{h \in [2^{K}]} \max_{j \in [n]} \mathbb{E}\{|R_{ij, h} - \mathbb{E}(R_{ij, h}|X_{j})|^{2}|X_{j}\} \leq \Lambda \|f\|_{\infty}. 
\end{align*}
Then, by the Bernstein inequality, with probability at least $1 - C\exp(-z)$, we have 
\begin{align*}
    \Gamma_{n, 1} \leq n \max_{h \in [2^{K}]} \max_{j \in [n]} \bigg|\sum_{i \neq j} \{R_{ij, h} - \mathbb{E} (R_{ij, h}|X_{j})\}\bigg|^{2} \leq C_{1}\{n (z + \log n)^{2} + n^{2} \Lambda \|f\|_{\infty} (z + \log n)\}. 
\end{align*}
We now bound $\Gamma_{n, 2}$. Observe that $\{|\mathbb{E}(R_{ij, h}|X_{j}) - \mathbb{E}(R_{ij, h})|^{2}\}_{j \in [n]}$ are independent random variables with 
\begin{align*}
    \max_{h \in [2^{K}]} \max_{j \in [n]} |\mathbb{E}(R_{ij, h}|X_{j}) - \mathbb{E}(R_{ij, h})|^{2} \leq \Lambda^{2} \|f\|_{\infty}^{2}.
\end{align*}
Hence, it follows from Lemma~\ref{Lemma_Hoeffding_inequality} that  
\begin{align*}
    \PP\left[\Gamma_{n, 2} \leq C\left\{n^{3} \Lambda^{2} \|f\|_{\infty}^{2} + n^{5/2} \Lambda^{2}  \|f\|_{\infty}^{2} \sqrt{(z + \log n)}\right\}\right] \geq 1 - \exp(-z).
\end{align*}
Recall that $W_{n}(t) = \{n(n - 1)\}^{-1}\sum_{i \neq j \in [n]} R(X_{i}, X_{j}, t)$. Hence $\bar{R}_{n, h} - \mathbb{E}(\bar{R}_{n, h}) = W_{n}(\Lambda h/2^{K}) - W_{n}(0)$ for each $h \in [2^{K}]$. Therefore, by Lemma~\ref{Lemma_Oscillation_W_process}, with probability at least $1 - C\exp(-z)$, we have 
\begin{align*}
    \Gamma_{n, 3} \leq C n^{3}\bigg\{\bigg(\frac{K + z}{n}\bigg)^{2} + \Lambda \|f\|_{\infty}\sqrt{\frac{z + c}{n}} + \frac{K + z}{n} \sqrt{\Lambda \|f\|_{\infty}}\bigg\}^{2}. 
\end{align*}
Putting all these pieces together, we obtain~\eqref{eq_upper_bound_Gamma}.
\end{proof}

\begin{proof}[Proof of Theorem~\ref{Theorem_quantile_consistency_one_Bootstrap}]
By Theorem~\ref{Theorem_quantile_consistency_one}, for any $\omega > 0$ such that $\omega \leq n \kappa_{0}^{2} c_{0}^{2}$, we have 
\begin{align*}
    \mathbb{P}\left(|\hat{\theta} - \theta| > \sqrt{\omega/(n\kappa_{0}^{2})}\right) \leq 2\exp(-\omega). 
\end{align*}
Combined with Lemma~\ref{Lemma_Oscillation_W_process}, for any $z > 0$ such that $z \leq 1/(4\|f\|_{\infty})$, with probability at least $1 - C_{1}\exp(-\omega)$, we have
\begin{align*}
    |U_{n}(\hat{\theta} &+ z) - U_{n}(\hat{\theta}) - \{U(\hat{\theta} + z) - U(\hat{\theta})\}|\cr
    &\leq |U_{n}(\hat{\theta} + z) - U_{n}(\theta) - \{U(\hat{\theta} + z) - U(\theta)\}|\cr
    &+ |U_{n}(\hat{\theta}) - U_{n}(\theta) - \{U(\hat{\theta}) - U(\theta)\}|\cr
    &\leq C_{2} \bigg(\|f\|_{\infty} z \sqrt{\frac{\omega + C}{n}} + \frac{K + \omega}{n} \sqrt{\|f\|_{\infty} z} + \frac{K + \omega}{n\kappa_{0}} \|f\|_{\infty}\bigg). 
\end{align*}
Similar to the proof of Theorem~\ref{Theorem_quantile_consistency_one}, for any $z \geq 0$ such that $\hat{\theta}^{\star} > \hat{\theta} + z$, we have 
\begin{align*}
    \frac{1}{2} &> U_{n}^{\star}(\hat{\theta} + z) \geq U_{n}(\hat{\theta} + z) - |U_{n}^{\star}(\hat{\theta} + z) - U_{n}(\hat{\theta} + z)|\cr
    &\geq U_{n}(\hat{\theta}) + \left\{U(\hat{\theta} + z) - U(\hat{\theta})\right\} - |U_{n}^{\star}(\hat{\theta} + z) - U_{n}(\hat{\theta} + z)|\cr
    &-|U_{n}(\hat{\theta} + z) - U_{n}(\hat{\theta}) - \{U(\hat{\theta} + z) - U(\hat{\theta})\}|\cr
    &\geq \frac{1}{2} + \frac{\kappa_{0}z}{2} - |U_{n}^{\star}(\hat{\theta} + z) - U_{n}(\hat{\theta} + z)| - C_{2} \left(\frac{K + \omega}{n} \sqrt{\|f\|_{\infty} z} + \frac{K + \omega}{n\kappa_{0}} \|f\|_{\infty}\right)\cr
    &\geq \frac{1}{2} + \frac{\kappa_{0}z}{2} - |U_{n}^{\star}(\hat{\theta} + z) - U_{n}(\hat{\theta} + z)| - \frac{2C_{2}(K + \omega)}{n\kappa_{0}} \|f\|_{\infty}. 
\end{align*}
Then it follows from Lemma~\ref{Lemma_Hoeffding_inequality} that
\begin{align*}
    \mathbb{P}^{\star} \left\{|\hat{\theta}^{\star} - \hat{\theta}| > \frac{2}{\kappa_{0}}\sqrt{\frac{z}{n}} + \frac{4C_{2}\|f\|_{\infty}(K + \omega)}{n\kappa_{0}^{2}}\right\} \leq 2 \exp(-C_{3}z) + \exp\left(-\frac{n}{16}\right). 
\end{align*}
By Lemma~\ref{Lemma_W_n_star_process} and Lemma~\ref{Lemma_concentration_Gamma_n}, with probability at least $1 - C\exp(-\omega)$, we have  
\begin{align*}
    \mathbb{P}^{\star}\bigg\{|U_{n}^{\star}(\hat{\theta}^{\star}) - U(\hat{\theta}^{\star}) - \{U_{n}^{\star}(\theta) - U(\theta)\}| > \frac{\mathcal{C}(K + z + \omega)}{n}\bigg\} \leq C \exp(-z) + \exp\left(-\frac{n}{16}\right),
\end{align*}
where $\mathcal{C}$ is a positive constant depending only on $\kappa_{0}, \|f\|_{\infty}$ and $c_{0}$. Therefore, combined with the fact that $\sup_{|\delta| \leq c_{1}} |U''(\theta + \delta)| \leq \kappa_{1}$, we obtain 
\begin{align*}
    \bigg|\hat{\theta}^{\star} - \theta - \frac{1}{B_{n} U'(\theta)} \sum_{i \neq j \in [n]} \omega_{i} \omega_{j} \left(\mathbb{I}\{\xi_{i} + \xi_{j} \leq 0\} - \frac{1}{2}\right)\bigg| \leq \frac{\mathcal{C}(K + z + \omega)}{n}.
\end{align*}
Combining these with Theorem~\ref{Theorem_quantile_consistency_one}, we obtain~\eqref{eq_Bahadur_bootstrap_one_sample}.  
\end{proof}

\begin{proof}[Proof of~\eqref{eq_CLT_bootstrap_one_sample}]
For each $i \in [n]$, denote 
\begin{align*}
    \mathcal{D}_{i} = \frac{2}{U'(\theta)}\bigg(\frac{1}{n - 1} \sum_{j \neq i} \mathbb{I}\{\xi_{i} + \xi_{j} \leq 0\} - U_{n}(\theta)\bigg). 
\end{align*}
Then if follows from~\eqref{eq_Bahadur_bootstrap_one_sample} that with probability at least $1 - C \exp(-\omega)$, we have 
\begin{align}
\label{eq_nonasymptotic_Bahadur_bootstrap_one_sample_log}
    \mathbb{P}^{\star}\bigg\{\bigg|\hat{\theta}^{\star} - \hat{\theta} - \frac{1}{n} \sum_{i = 1}^{n} \mathcal{D}_{i}(\omega_{i} - 1)\bigg| > \frac{C(\omega + \log n)}{n}\bigg\} \leq Cn^{-c}.
\end{align}
By Lemma~\ref{Lemma_sample_sigma_consistency}, we have $\mathbb{P}(\hat{\sigma}_{\theta} \geq \sigma_{\theta}/2) \geq 1 - C\exp(-n)$. Consequently, by Berry-Esseen theorem, with probability at least $1 - C\exp(-n)$, we have  
\begin{align*}
    \sup_{z \in \RR} \bigg|\PP^{\star}\bigg(\frac{1}{\sqrt{n} \hat{\sigma}_{\theta}} \sum_{i = 1}^{n} \mathcal{D}_{i} (\omega_{i} - 1) \leq z\bigg) - \Phi(z)\bigg| \leq \frac{C}{\sqrt{n}}.
\end{align*}
Combining this with~\eqref{eq_nonasymptotic_Bahadur_bootstrap_one_sample_log} yields~\eqref{eq_CLT_bootstrap_one_sample}.  
\end{proof}

\subsection{Proof of Theorem~\ref{Theorem_Bootstrap_nonasymptotic_two_sample}}
Based on Theorem~\ref{Theorem_quantile_consistency_two}, the proof of Theorem~\ref{Theorem_Bootstrap_nonasymptotic_two_sample} can be derived by following the proof of Theorem~\ref{Theorem_quantile_consistency_one_Bootstrap}. Thus we omit the details.

\section{Proof of Results in \S\ref{large-scale-one}}
\subsection{Proof of Theorem~\ref{Theorem_GA_one}}
\begin{proof}[Proof of Theorem~\ref{Theorem_GA_one}]
For each $i \in [n]$ and $\ell \in [p]$, denote $D_{i\ell} = \{1 - 2F_{\ell}(-\xi_{i\ell})\}/U_{\ell}'(\theta_{\ell})$. Define $\bar{\bD}_{n} = (\bar{D}_{n1}, \ldots, \bar{D}_{np})^{\top} \in \RR^{p}$, where $\bar{D}_{n\ell} = n^{-1} \sum_{i = 1}^{n} D_{i\ell}$ for each $\ell \in [p]$. Note that $|D_{i\ell}| \leq 1/\kappa_{0}$ uniformly for $i \in [n]$ and $\ell \in [p]$. Hence, it follows from Theorem 2.1 in~\citet{CCKK2022} that  
\begin{align*}
    \sup_{z \in \RR} |\PP(\sqrt{n}\|\bar{\bD}_{n}\|_{\infty} \leq z) - \PP(\|\bZ\|_{\infty} \leq z)| \leq \frac{C\log^{5/4}(pn)}{n^{1/4}}. 
\end{align*}
Let $\delta^{\sharp} = C \log(pn)/\sqrt{n}$. By Theorem~\ref{Theorem_quantile_consistency_one}, it follows that 
\begin{align*}
    \sup_{z \in \RR} &|\PP(\sqrt{n}\|\hat{\btheta} - \btheta\|_{\infty} \leq z) - \PP(\sqrt{n} \|\bar{\bD}_{n}\|_{\infty} \leq z)|\cr
    &\leq \PP(\sqrt{n}\|\hat{\btheta} - \btheta - \bar{\bD}_{n}\|_{\infty} > \delta^{\sharp}) + \sup_{z \in \RR} \PP(z < \sqrt{n}\|\bar{\bD}_{n}\|_{\infty} \leq z + \delta^{\sharp})\cr
    &\lesssim n^{-c} + \delta^{\sharp} \sqrt{\log p} + \frac{\log^{5/4}(pn)}{n^{1/4}} \leq \frac{C\log^{5/4}(pn)}{n^{1/4}}.
\end{align*}
Putting all these pieces together, we obtain~\eqref{eq_GA_one_sample}.
\end{proof}

\subsection{Proof of Theorem~\ref{Theorem_Bootstrap_consistency_high_dimensional_GA_one_sample}}
\begin{proof}[Proof of Theorem~\ref{Theorem_Bootstrap_consistency_high_dimensional_GA_one_sample}]
For each $\ell \in [p]$ and $i \in [n]$, denote 
\begin{align*}
    \mathcal{D}_{i\ell} = \frac{2}{\sigma_{\ell}U_{\ell}'(\theta_{\ell})}\bigg(\frac{1}{n - 1}\sum_{j\neq i} \mathbb{I}\{\xi_{i\ell} + \xi_{j\ell} \leq 0\} - U_{n\ell}(\theta_{\ell})\bigg). 
\end{align*}
Let $\bar{\bm{\mathcal{D}}}_{n} = (\bar{\mathcal{D}}_{n1}, \ldots, \bar{\mathcal{D}}_{np})^{\top}$, where $\bar{\mathcal{D}}_{n\ell} = n^{-1} \sum_{i = 1}^{n} \mathcal{D}_{i\ell}$ for each $\ell \in [p]$. By the triangle inequality, 
\begin{align*}
    \max_{\ell \in [p]} \sum_{i = 1}^{n}(D_{i\ell} &- \mathcal{D}_{i\ell})^{2} \lesssim \max_{\ell \in [p]} \sum_{i = 1}^{n} \bigg(\frac{1}{n - 1} \sum_{j \neq i} \mathbb{I}\{\xi_{i\ell} + \xi_{j\ell} \leq 0\} - F_{\ell}(-\xi_{i\ell})\bigg)^{2} + n \max_{\ell \in [p]} \bigg|U_{n\ell}(\theta_{\ell}) - \frac{1}{2}\bigg|^{2}\cr
    &\leq n \max_{\ell \in [p]} \max_{i \in [n]} \bigg|\frac{1}{n - 1} \sum_{j \neq i} \mathbb{I}\{\xi_{i\ell} + \xi_{j\ell} \leq 0\} - F_{\ell}(-\xi_{i\ell})\bigg|^{2} + n \max_{\ell \in [p]} \bigg|U_{n\ell}(\theta_{\ell}) - \frac{1}{2}\bigg|^{2}\cr
    &=: \Delta_{D}^{\diamond} + \Delta_{D}^{\circ}. 
\end{align*}
By Lemma~\ref{Lemma_Hoeffding_inequality}, it follows that $\PP\{\Delta_{D}^{\diamond} > C\log (np)\} \lesssim (np)^{-c}$. Recall that
\begin{align*}
    U_{\ell}(\theta_{\ell}) = \frac{1}{n(n - 1)} \sum_{i \neq j \in [n]} \mathbb{I}\{\xi_{i\ell} + \xi_{j\ell} \leq 0\}, \enspace \ell \in [p]. 
\end{align*}
Hence, it follows from Lemma~\ref{Lemma_Hoeffding_inequality_U_statistics} that $\PP\{\Delta_{D}^{\circ} > C\log(np)\} \lesssim (np)^{-c}$. Consequently, for any $\delta > 0$, by Theorem 1.1 in~\citet{bentkus2015tight}, with probability at least $1 - C(np)^{-c}$, we have 
\begin{align*}
    \PP^{\star}\bigg(\max_{\ell \in [p]} \bigg|\sum_{i = 1}^{n}(D_{i\ell} - \mathcal{D}_{i\ell})(\omega_{i} - 1)\bigg| > \delta\bigg) \lesssim p \exp\left(-\frac{C\delta^{2}}{\log (np)}\right).
\end{align*}
Combined with Lemma, with probability at least $1 - C(np)^{-c}$, we have 
\begin{align*}
    \sup_{z \in \RR} &\bigg|\PP^{\star}\bigg(\max_{\ell \in [p]} \bigg|\sum_{i = 1}^{n} D_{i\ell}(\omega_{i} - 1)\bigg| \leq z\bigg) - \PP^{\star}\bigg(\max_{\ell \in [p]} \bigg|\sum_{i = 1}^{n} \mathcal{D}_{i\ell} (\omega_{i} - 1)\bigg| \leq z\bigg)\bigg|\cr
    &\lesssim p \exp\{-C \log(np)\} + \frac{\log (np)\sqrt{\log p}}{\sqrt{n}} \leq \frac{\log (np)\sqrt{\log p}}{\sqrt{n}}. 
\end{align*}
\end{proof}

\subsection{Proof of Theorem~\ref{Theorem_FDP_consistency}}
\begin{proof}[Proof of Theorem~\ref{Theorem_FDP_consistency}]
Following the proof of Theorem~\ref{Theorem_Cramer_one}, it follows from Theorem~\ref{Theorem_quantile_consistency_one_Bootstrap} that with probability at least $1 - Cn^{-c}$, we have 
\begin{align}
\label{eq_Moderate_deviation_uniform}
    \frac{\mathbb{P}^{\star}\{\sqrt{n}(\hat{\theta}_{\ell}^{\star} - \hat{\theta}_{\ell}) > z\}}{2 \{1 - \Phi(z/\sigma_{\ell})\}} = 1 + o(1), 
\end{align}
uniformly for $0 < z\leq o(n^{1/6})$ and $\ell \in [p]$. By Lemma 1 in~\citet{Storey2004}, it follows that 
\begin{align*}
    t_{S} = \left\{t \in [0, 1] : t \leq \frac{\alpha \max(\sum_{\ell = 1}^{p} \mathbb{I}\{P_{\ell} \leq t\}, 1)}{p}\right\}.
\end{align*}
By the definition of $t_{S}$, we have 
\begin{align}\label{eq_tS}
    t_{S} = \frac{\alpha\max(\sum_{\ell = 1}^{p} \mathbb{I}\{P_{\ell} \leq t_{S}\}, 1)}{p}.
\end{align}
Observe that $t_{S} \geq \alpha/p$. Hence, it follows from~\eqref{eq_Moderate_deviation_uniform} that $\mathbb{P}\{t_{S} > \max_{\ell \in [p]} \mathcal{G}_{\ell}^{\star}(\sigma_{\ell}(2\log p)^{1/2})\} \to 1$. Combining this with~\eqref{eq_tS} yields
\begin{align*}
    t_{S} \geq \frac{\alpha}{p} \sum_{\ell = 1}^{p} &\mathbb{I}\{P_{\ell} \leq t_{S}\} \geq \frac{\alpha}{p} \sum_{\ell = 1}^{p} \mathbb{I}\left\{\mathcal{G}_{\ell}^{\star}(\sqrt{n}|\hat{\theta}_{\ell}|) \leq \mathcal{G}_{\ell}^{\star}\left(\sigma_{\ell}\sqrt{2\log p}\right)\right\}\cr
    &\geq \frac{\alpha}{p} \sum_{\ell = 1}^{p} \mathbb{I}\left\{\sqrt{n}|\hat{\theta}_{\ell}| \geq \sigma_{\ell}\sqrt{2\log p}\right\}\cr
    &\geq \frac{\alpha}{p} \sum_{\ell = 1}^{p} \mathbb{I}\bigg\{\frac{|\theta_{\ell}|}{\sigma_{\ell}} \geq \sqrt{2\log p} + \max_{\ell \in [p]} \frac{\sqrt{n}(\hat{\theta}_{\ell} - \theta_{\ell})}{\sigma_{\ell}}\bigg\}.
\end{align*}
For some $\lambda > 0$, define 
\begin{align*}
    \mathcal{E}_{\lambda} = \bigg\{\sqrt{n} \max_{\ell \in [p]} \frac{|\hat{\theta}_{\ell} - \theta_{\ell}|}{\sigma_{\ell}} \leq (1 + \lambda) \sqrt{2\log p}\bigg\}.
\end{align*}
Under $\mathcal{E}_{\lambda}$, it is straightforward to verify that 
\begin{align*}
    t_{S} \geq \frac{\alpha}{p} \sum_{\ell = 1}^{p} \mathbb{I}\left\{\frac{|\theta_{\ell}|}{\sigma_{\ell}} \geq (2 + \lambda)\sqrt{2\log p}\right\}
\end{align*}
By theorem~\ref{Theorem_Cramer_one}, it follows that $\mathbb{P}(\mathcal{E}_{\lambda}^{c}) \leq C p \exp\{-(1 + \lambda)^{2} \log p\} \leq C p^{-\lambda^{2}-2\lambda}$. Then, following the proof of Theorem 3.3 in~\citet{Zhou2018}, for any sequence $1\leq b_{p} \to \infty$, it is straightforward to derive that 
\begin{align*}
    \sup_{b_{p}/p \leq t \leq 1} \bigg|\frac{\sum_{\ell \in \mathcal{H}_{0}}\mathbb{I}\{P_{\ell} \leq t\}}{|\mathcal{H}_{0}|t} - 1\bigg| \overset{\PP}{\to} 0.
\end{align*}
Putting all these pieces together, we obtain~\eqref{eq_FDP_one_sample_convergence}. 
\end{proof}


\subsection{Large-scale Two-sample Simultaneous Testing}
\label{subsection_two_sample_simultaneous_testing}
We consider the following high dimensional two-sample global test, 
\begin{align}
\label{global_test_two}
    H_{0} : \theta_{\ell} = \theta_{\ell}^{\circ} \enspace \mathrm{for\ all} \enspace \ell \in [p] \enspace \mathrm{versus} \enspace H_{1} : \theta_{\ell} \neq \theta_{\ell}^{\circ} \enspace \mathrm{for\ some} \enspace \ell \in [p].
\end{align}
Given the HL estimators $\hat{\Theta}_{\ell} = \mathrm{median}\{X_{i\ell} - Y_{j\ell} : i \in [n], j \in [m]\}$, $\ell \in [p]$, for $\Theta_{\ell} = \theta_{\ell} - \theta_{\ell}^{\circ}$, we shall the null hypothesis $H_{0}$ in~\eqref{global_test_two} whenever $\max_{\ell \in [p]} |\hat{\Theta}_{\ell}|$ exceeds some threshold. To this end, we develop a Gaussian approximation for $\max_{\ell \in [p]} |\hat{\Theta}_{\ell} - \Theta_{\ell}|$. More specifically, let $\bm{\mathcal{Z}} = (\mathcal{Z}_{1}, \ldots, \mathcal{Z}_{p})^{\top}$ be a~\emph{p}-dimensional centered Gaussian random vector with 
\begin{align*}
    \Cov(\mathcal{Z}_{k}, \mathcal{Z}_{\ell}) = \frac{N \Cov(\bar{G}_{n k} - \bar{F}_{mk}, \bar{G}_{n\ell} - \bar{F}_{m\ell})}{\mathcal{U}_{k}'(\Theta_{k})\mathcal{U}_{\ell}'(\Theta_{\ell})}, \enspace k, \ell \in [p]. 
\end{align*}
In the following theorem, we establish a non-asymptotic upper bound for the Kolmogorov distance between the distribution functions of $\max_{\ell \in [p]} |\hat{\Theta}_{\ell} - \Theta_{\ell}|$ and its Gaussian analogue $\max_{\ell \in [p]} |\mathcal{Z}_{\ell}|$.

\begin{theorem}
\label{Theorem_two_sample_GA}
Let $\mathcal{U}_{\ell}(t) = \mathbb{P}(X_{1\ell} - Y_{1\ell} \leq t)$ for each $\ell \in [p]$. Assume that there exist positive constants $\bar{c}_{0}, \bar{\kappa}_{0}, \bar{c}_{1}$ and $\bar{\kappa}_{1}$ such that  
\begin{align*}
    \min_{\ell \in [p]} \inf_{|\Lambda| \leq \bar{c}_{0}} \cU_{\ell}'(\Theta_{\ell} + \Lambda) \geq \bar{\kappa}_{0} > 0 \enspace \mathrm{and} \enspace \max_{\ell \in [p]} \sup_{|\Lambda| \leq \bar{c}_{1}} |\cU_{\ell}''(\Theta_{\ell} + \Lambda)| \leq \bar{\kappa}_{1}. 
\end{align*}
Then, under Assumption~\ref{Assumption_two_sample_size}, we have 
\begin{align*}
    \sup_{z \in \RR} \left|\PP\left(\max_{\ell \in [p]} \sqrt{N}|\hat{\Theta}_{\ell} - \Theta_{\ell}| \leq z\right) - \PP\left(\max_{\ell \in [p]}|\mathcal{Z}_{\ell}| \leq z\right)\right| \leq \frac{C\log^{5/4}(Np)}{N^{1/4}}. 
\end{align*}
\end{theorem}

\begin{proof}[Proof of Theorem~\ref{Theorem_two_sample_GA}]
We are able to prove Theorem~\ref{Theorem_two_sample_GA} by using similar arguments as in Theorem~\ref{Theorem_GA_one}. The major differences only lies in changing the one-sample to two-sample estimators. Thus, we decided to  omit the corresponding details.
\end{proof}

Motivated by Theorem~\ref{Theorem_two_sample_GA}, an asymptotic $\alpha$-level test for~\eqref{global_test_two} is given by $\mathbb{I}\{\max_{\ell \in [p]}|\hat{\Theta}_{\ell}| > Q_{1 - \alpha}^{\diamond}\}$, where $Q_{1 - \alpha}^{\diamond}$ stands for the $(1 - \alpha)$th quantile of the bootstrap statistic $\max_{\ell \in [p]} |\hat{\Theta}_{\ell}^{\star} - \hat{\Theta}_{\ell}|$, namely, 
\begin{align*}
    Q_{1 - \alpha}^{\diamond} = \inf\left\{z \in \mathbb{R} : \mathbb{P}^{\star}\left(\max_{\ell \in [p]} |\hat{\Theta}_{\ell}^{\star} - \hat{\Theta}_{\ell}| \leq z\right) \geq 1 - \alpha\right\}, \enspace \alpha \in (0, 1). 
\end{align*}
The validity of the proposed test is justified via the following theorem.

\begin{theorem}
\label{Theorem_Bootstrap_consistency_high_dimensional_GA_two_sample}
Under the conditions of Theorem~\ref{Theorem_two_sample_GA}, we have 
\begin{align*}
    \left|\mathbb{P}\left(\max_{\ell \in [p]} |\hat{\Theta}_{\ell} - \Theta_{\ell}| > Q_{1 - \alpha}^{\diamond}\right) - \alpha\right| \leq \frac{C \log^{5/4}(pN)}{N^{1/4}}. 
\end{align*}
\end{theorem}

\begin{proof}[Proof of Theorem~\ref{Theorem_Bootstrap_consistency_high_dimensional_GA_two_sample}]
The proof of Theorem~\ref{Theorem_Bootstrap_consistency_high_dimensional_GA_two_sample} can be derived similarly by following the proof procedure of Theorem~\ref{Theorem_Bootstrap_consistency_high_dimensional_GA_one_sample} by replacing the one-sample estimator with the two-sample version. Therefore, we omit the details. 
\end{proof}




\subsection{Proof of Theorem~\ref{Theorem_FDP_consistency_2}}
The proof of Theorem~\ref{Theorem_FDP_consistency_2} can be derived similarly by following the proof procedure of Theorem~\ref{Theorem_FDP_consistency} by changing the one-sample estimator to the two-sample version. Therefore, we omit the details.

\end{document}